\documentclass[12pt]{article}
\usepackage{amsmath,amsthm,amsfonts,amssymb}
\usepackage[all]{xy}
\usepackage[usenames]{color} 

\newcounter{enumitemp}
\newenvironment{enumeratecontinue}{
 \setcounter{enumitemp}{\value{enumi}}
 \begin{enumerate}
 \setcounter{enumi}{\value{enumitemp}}
}
{
 \end{enumerate}
}

\newcommand\pref[1]{(\ref{#1})}

\newtheorem{thm}{Theorem}[section]

\newtheorem{theorem}[thm]{Theorem}

\newtheorem{lemma}[thm]{Lemma}

\newtheorem*{theoremD}{Global WWPD Theorem}
\newtheorem*{corollary*}{Corollary}

\newtheorem{corollary}[thm]{Corollary}
\newtheorem{proposition}[thm]{Proposition}
\newtheorem{PropAndDef}[thm]{Proposition And Definition}
\newtheorem*{proposition*}{Proposition}

\newtheorem*{theoremWWPDSpecial}{Theorem \ref{ThmWWPDHTwoBSpecial}}

\theoremstyle{definition}

\newtheorem{definition}[thm]{Definition} 

\newtheorem*{defn*}{Definition}

\theoremstyle{remark}

\newcounter{remarks}
{\paragraph*{Remarks}\smallskip
 \begin{list}{\arabic{remarks}. }{\usecounter{remarks}%
 \setlength{\leftmargin}{0in}%
 \setlength{\rightmargin}{0in}%
 \setlength{\labelsep}{0pt}%
 \setlength{\labelwidth}{0pt}%
 \setlength{\listparindent}{0pt}%
 }
}
{
\end{list}
}

\newcommand\vp{{\vphantom\prime}}

\newcommand\from\colon
\newcommand\inv{{-1}}
\newcommand\subgroup{<}
\newcommand\normal{\triangleleft}
\newcommand\infinity\infty

\newcommand\disjunion\coprod
\newcommand\act\curvearrowright

\newcommand\lessless{<\!\!<}

\DeclareMathOperator\image{Image}
\DeclareMathOperator\kernel{Ker}

\DeclareMathOperator\Sym{Sym}

\DeclareMathOperator\Isom{Isom}

\newcommand{\R}{{\mathbb R}}
\newcommand\reals{\R}

\newcommand{\Z}{{\mathbb Z}}

\newcommand\semidirect{\rtimes}

\DeclareMathOperator{\Out}{\mathsf{Out}}
\DeclareMathOperator{\Aut}{\mathsf{Aut}}
\DeclareMathOperator{\Inn}{\mathsf{Inn}}

\DeclareMathOperator{\Stab}{\mathsf{Stab}}

\newcommand\cbdy\delta

\newcommand{\A}{\mathcal A}

\newcommand{\noneg}{NEG}
\renewcommand\neg\noneg

\newcommand{\wt}{\widetilde}

\newcommand{\comment}[1]{}

\newcommand\BeFujiTag{BestvinaFujiwara:bounded}
\newcommand\BeFuji{\cite{BestvinaFujiwara:bounded}}

\newcommand\bdy\partial

\newcommand\intersect\cap
\newcommand\union\cup
\newcommand\<\langle
\renewcommand\>\rangle
\newcommand\meet\wedge
\newcommand\composed{\circ}
\newcommand\cross\times
\newcommand\restrict{\bigm |}

\newcommand\inject\hookrightarrow
\DeclareMathOperator\Length{Length}
\newcommand\abs[1]{\left|#1\right|}
\newcommand\Id{\text{Id}}

\newcommand\injectto\hookrightarrow
\newcommand\injectfrom\hookleftarrow
\newcommand\surjectto\twoheadrightarrow
\newcommand\surjectfrom\twoheadleftarrow

 \newcommand\surjection\twoheadrightarrow
 
\newcommand\suchthat{\bigm|}

\title{Second bounded cohomology and WWPD}

\author{Michael Handel and Lee Mosher \thanks{The first author  was supported by the National Science Foundation under Grant No.~DMS-1308710 and by PSC-CUNY under grants in Program Years 46 and 47. The second author was supported by the National Science Foundation under Grant No.~DMS-1406376.}}

\begin{document}

\maketitle

\begin{abstract}
Given a group acting on a Gromov hyperbolic space, Bestvina and Fujiwara introduced the WPD property --- weak proper discontinuity --- for studying the 2nd bounded cohomology of the group. We carry out a more general study of second bounded cohomology using a \emph{really} weak property discontinuity property known as WWPD that was introduced by Bestvina, Bromberg, and Fujiwara. 
\end{abstract}

\section{Introduction}

In their work on the 2nd bounded cohomology of subgroups of surface mapping class groups \BeFuji, Bestvina and Fujiwara introduced the WPD property of a group action on a Gromov hyperbolic space. The WPD property functions as an abstraction of methods that evolved over a series of works which study the 2nd bounded cohomology of groups that act in certain ways on hyperbolic spaces \cite{Brooks:H2bRemarks}, \cite{BrooksSeries:H2bSurface}, \cite{EpsteinFujiwara}, \cite{Fujiwara:H2BHyp}, \cite{Fujiwara:H2bFreeProduct}. In particular, Theorems~7 and~8 of \BeFuji\ show that when all the loxodromic elements of a hyperbolic group action $\Gamma \act X$ satisfy the WPD property then the second bounded cohomology group $H^2_b(\Gamma;\reals)$ is of uncountable dimension, by virtue of containing an embedded copy of $\ell^1$. The WPD property itself has also proved useful for other applications \cite{Bowditch:tight,Osin:AcylHyp,BHS:Hierarchy}. 

In their recent work on mapping class groups, Bestvina, Bromberg and Fujiwara introduced a \emph{really} weak version of WPD known as WWPD, applying it to stable commutator length \cite{BBF:SCLonMCG}. In this paper we apply WWPD properties of a hyperbolic group action $\Gamma \act X$ to the study of $H^2_b(\Gamma;\reals)$, generalizing the WPD methods of \BeFuji. Our main result is the Global WWPD Theorem, stated below. We also include, in Section~2, a general study of the WWPD property for an individual element $g \in \Gamma$, including proofs of equivalence of several different versions of that property,  such as the ``WWPD~(2)'' property which says that the repeller--attractor pair $\bdy_\pm g = (\bdy_- g,\bdy_+ g)$ has discrete orbit under the natural action of $\Gamma$ on $\bdy X \times \bdy X - \Delta$.


Our motivation for proving the Global WWPD Theorem is its application to studying the 2nd bounded cohomology of subgroups of $\Out(F_n)$, the outer automorphism group of a rank $n$ free group $F_n$. The main theorem of our two part work \cite{HandelMosher:BddCohomologyI,HandelMosher:BddCohomologyII}\footnote{This paper was split off from an earlier version (\texttt{arXiv:1511.06913v4}) of our preprint \cite{HandelMosher:BddCohomologyI}.}
says that for every finitely generated subgroup $\Gamma \subgroup \Out(F_n)$, either $\Gamma$ has an abelian subgroup of finite index or its second bounded cohomology $H^2_b(\Gamma;\reals)$ contains an embedded copy of $\ell^1$. In that work, although the group actions we considered did not have enough WPD elements to prove the main theorem by applying the WPD methods of \BeFuji, nonetheless we \emph{were} able to construct sufficiently many WWPD elements to allow us to apply the Global WWPD Theorem. 

The Global WWPD Theorem can be thought of as a simultaneous generalization of Theorems~7 and~8 of \BeFuji. In each of those theorems the conclusion is that $H^2_b(\Gamma;\reals)$ contains an embedded copy of $\ell^1$. Theorem~7 of \BeFuji\ reaches this conclusion assuming the existence of a hyperbolic action of the group $\Gamma$ which satisfies a strong ``global WPD hypothesis'', saying that the action is nonelementary and that every loxodromic element of $\Gamma$ satisfies WPD. Theorem~8 of \BeFuji\ weakens that hypothesis by requiring not an action of $\Gamma$ itself, but instead just a hyperbolic action of a finite index normal subgroup $N \normal \Gamma$ satisfying the global WPD hypothesis, and in addition satisfying a strong wreath product hypothesis. 

The Global WWPD Theorem will also require only a hyperbolic action of a finite index normal subgroup $N \normal \Gamma$. No wreath product hypothesis is needed; in its place the proof uses the Kaloujnin Krasner wreath product embedding. But some kind of replacement for the global WPD hypothesis is needed, asserting the existence of sufficiently many WWPD elements (see the introductory paragraphs of Section~\ref{SectionWreathProductMethods} for further discussion). For purposes of application to \cite{HandelMosher:BddCohomologyI,HandelMosher:BddCohomologyII}, we were able to arrive at a rather weak but sufficient hypothesis by abstracting features of the proofs of \cite[Theorems~7 and~8]{\BeFujiTag}, as follows:
\begin{description}
\item[Global WWPD Hypothesis:] Given a group $\Gamma$ and a hyperbolic action $N \act X$ of a finite index normal subgroup $N \normal \Gamma$, we say that the \emph{global WWPD hypothesis} holds if there exists a rank~2 free subgroup $F < N$ satisfying the following: 
\begin{itemize}
\item Each element of $F$ is a WWPD element of the action $N \act X$, and the restricted action $F \act X$ is Schottky;
\item For every $g \in \Gamma$, denoting its associated inner automorphism as \hbox{$i_g \from \Gamma \to \Gamma$,} and letting $N \act_g X$ denote the composed action $N \xrightarrow{i_g} N \act X$, the restricted action $F \act_g X$ satisfies one of two properties: either the elements of $F$ are all elliptic; or each element of $F$ is a WWPD element of the action $N \act_g X$ and the action $F \act_g X$ is Schottky. 
\end{itemize}
\end{description}
With this hypothesis we can now state our main theorem:

\begin{theoremD}
For any group $\Gamma$ and any hyperbolic action $N \act X$ of a finite index normal subgroup $N \normal \Gamma$, if the global WWPD hypothesis holds then $H^2_b(\Gamma;\reals)$ contains an embedding of $\ell^1$.
\end{theoremD}
\noindent
The proof of the Global WWPD Theorem, which is found in Section~\ref{SectionWreathProductMethods}, uses WWPD methods to construct quasimorphisms, generalizing constructions using WPD methods that are found in \BeFuji. 

In the special case $N=\Gamma$, i.e.\ the case of a hyperbolic action of the whole group $\Gamma$, the statement and proof of the Global WWPD theorem can be considerably simplified: the global WWPD hypothesis can be replaced with the assumption that a single WWPD element exists; and wreath product methods are not needed in the proof. The resulting theorem is as follows:

\begin{theoremWWPDSpecial}
If $\Gamma \act X$ is a hyperbolic action possessing an independent pair of loxodromic elements, and if $\Gamma$ has a WWPD element, then $H^2_b(\Gamma;\reals)$ contains an embedded~$\ell^1$. \qed
\end{theoremWWPDSpecial}
\noindent
The proof is given at the end of Section~\ref{SectionWWPDTheory} by combining the tools from that section with a result of Bestvina and Fujiwara, namely \cite[Theorem 1]{\BeFujiTag}.


\section{Weakening weak proper discontinuity}
\label{SectionWWPDTheory}
This section develops the theory of WWPD elements of hyperbolic actions. 

Section~\ref{SectionBasicConcepts} contains a review of basic concepts of hyperbolic actions, and of the close relation between 2nd bounded cohomology and quasimorphisms. 

Section~\ref{SectionWPDandWWPD} reviews the WPD and WWPD properties, and contains proofs of some key results, namely Proposition~\ref{PropDefWWPD} and its corollaries, which state several equivalent formulations of the WWPD property of an element of a hyperbolic action, and which establish various other properties of such elements. 

In Section~\ref{SectionSim} we describe some tools that will be used in the proof of the global WWPD Theorem to be presented in Section~\ref{SectionWreathProductMethods}. Definition~\ref{DefinitionSim} reviews the useful equivalence relation $g \sim h$ on the set of loxodromic elements of a hyperbolic action, which was introduced by Bestvina and Fujiwara for purposes of their general quasimorphism construction theorem \cite[Theorem 1]{\BeFujiTag}. Lemma~\ref{LemmaSimEquivProps} explains some elementary interactions between the WWPD property and the equivalence relation $g \sim h$. Then we review hyperbolic ping-pong and Schottky groups, followed by Proposition~\ref{PropWWPDProps} which explains how to use WWPD to construct Schottky subgroups possessing useful inequivalent loxodromic pairs $g \not\sim h$ (c.f.\ \cite[Proposition 6 (5)]{\BeFujiTag}).

\subsection{Basic concepts.}
\label{SectionBasicConcepts}

\paragraph{Hyperbolic actions.} We use $\A \from \Gamma \act X$ to denote an action of a group $\Gamma$ on an object $X$, meaning simply a homomorphism $\A \from \Gamma \to \Isom(X)$ to the group of self-isomorphisms of~$X$; this applies in any category. We often write $\Gamma \act X$, suppressing~$\A$, and we write expressions such as $\gamma \cdot x$ as a shorthand for $\A(\gamma)(x)$, \hbox{($\gamma \in \Gamma$, $x \in X$).} The \emph{stabilizer} of a subset $Y \subset X$ is denoted $\Stab(Y) = \Stab(Y;\Gamma) = \{\gamma \in \Gamma \suchthat \gamma \cdot Y = Y\}$.

For background material on quasi-isometries and Gromov hyperbolicity, see for example \cite{Gromov:hyperbolic,CDP,ABCFLMSS,GhysHarpe:afterGromov}. Here we review some notation, terminology, and simple facts. For the rest of the paper, we will use many elementary facts about this material without citation.

Let $X$ be a geodesic metric space which is hyperbolic in the sense of Gromov. Adjoining the Gromov boundary $\bdy X$ gives a topological space $\overline X = X \union \bdy X$ in which $X$ is dense. Any continuous quasi-isometric embedding $X \mapsto Y$ of hyperbolic metric spaces extends uniquely to a continuous map $\overline X \mapsto \overline Y$. A \emph{hyperbolic action} is simply an isometric action $\Gamma \act X$ of a group $\Gamma$ on a hyperbolic metric space $X$; 
the unique extensions to $\overline X$ of the individual elements of $\Gamma$, taken together, form a unique extended topological action $\Gamma \act \overline X$. The space of two point subsets of the boundary is denoted $\bdy^2 X = \{\{\xi,\eta\} \suchthat \xi \ne \eta \in \bdy X\}$ with topology induced by the 2-1 covering map $(\bdy X \times \bdy X) - \Delta \mapsto \bdy^2 X$, equivalently the Hausdorff topology on compact subsets of~$\bdy X$. For any isometric action $\Gamma \act X$, the extended action on $\bdy X$ naturally induces actions of $\Gamma$ on $(\bdy X \times \bdy X) - \Delta$ and on $\bdy^2 X$.

Isometries $h \from X \to X$ of a hyperbolic metric space are classified into three types: $h$ is \emph{elliptic} if the orbit $h^k(x)$ is bounded for some (every) $x \in X$; $h$ is \emph{loxodromic} if the orbit map $k \mapsto h^k \cdot x$ is a continuous quasi-isometric embedding for some (every) $x \in X$; and otherwise $h$ is \emph{parabolic}. 

The loxodromic property has several equivalent formulations. For example, $h$ is loxodromic if and only if there exists an ordered pair of distinct boundary points denoted $\bdy_\pm h = (\bdy_- h, \bdy_+ h) \in \bdy X \times \bdy X - \Delta$ such that for all $x \in \overline X$, if $x \ne \bdy_- h$ then the $\lim_{k \to +\infty} h^k(x)=\bdy_+ h$, and if $x \ne \bdy_+ h$ then $\lim_{k \to -\infty} h^k(x) = \bdy_- h$. The associated unordered pair of points is denoted $\bdy h = \{\bdy_- h,\bdy_+ h\} \in \bdy^2 X$. Also, $h$ is loxodromic if and only if it has a \emph{quasi-axis}, which is a continuous quasi-isometric embedding $\ell \from \reals = (-\infty,+\infty) \to X$ such that $h(\ell(t))=\ell(t+T)$ where $T > 0$ is independent of $t \in \reals$. In this situation the map $\ell$ has a unique continuous extension $\bar \ell \from [-\infty,+\infty] \to \overline X$, this extension satisfies $\bar \ell(-\infty) = \bdy_- h$ and $\bar \ell(+\infty) = \bdy_+ h$, and we denote $\bdy_-\ell = \bdy_- h$ and $\bdy_+ \ell = \bdy_+ h$. A quasi-axis $\ell$ is determined by any of its \emph{fundamental domains}, meaning its restrictions of the form $\ell\restrict [a,a+T]$ for any $a \in \reals$. 

\paragraph{$2^{\text{nd}}$ bounded cohomology.} Consider a group $\Gamma$ and its cochain complex $C^*(\Gamma;\reals)$, where $C^k(\Gamma;\reals)$ is the vector space of functions $f \from \Gamma^k \mapsto \reals$, with the standard coboundary operator $\delta \from C^k(\Gamma;\reals) \to C^{k+1}(\Gamma;\reals)$ \cite{Brown:cohomology}. The only case that we shall actually use is $k=1$ where $\delta \from C^1(\Gamma;\reals) \to C^2(\Gamma;\reals)$ is defined by 
$$\delta f (g_0,g_1) = f(g_0) - f(g_0 g_1) + f(g_1) 
$$
For general $k$ the formula for $\delta$ is given by
$$\delta f (g_0,g_1,\ldots,g_k) = \sum_{i=0}^{k+1} (-1)^i f(\pi_i(g_0,g_1,\ldots,g_k))
$$
where $\pi_0$ omits the $g_0$ coordinate, $\pi_{k+1}$ omits the $g_k$ coordinate, and if $1 \le i \le k$ then $\pi_i$ removes the comma between the coordinates $g_{i-1}$ and $g_i$ and multiplies them together. 

The bounded cohomology $H^k_b(\Gamma;\reals)$ is the cohomology of the subcomplex of bounded cochains $C^k_b(\Gamma;\reals) = \{f \in C^k(\Gamma,\reals) \suchthat \text{$f$ is bounded}\}$, in particular:
$$H^2_b(\Gamma;\reals) = \frac{\kernel(C^2_b(\Gamma;\reals) \to C^3_b(\Gamma;\reals))}{\image(C^1_b(\Gamma;\reals) \to C^2_b(\Gamma;\reals))}
$$

For studying the $2^{\text{nd}}$ bounded cohomology $H^2_b(\Gamma;\reals)$, we review a standard method which uses the closely related concept of ``quasimorphisms'' on $\Gamma$ (see for example the works cited in the first paragraph of the introduction). 

The vector space of \emph{quasimorphisms} of $\Gamma$ is defined to be
$$QH(\Gamma;\reals) = \{f \in C^1(\Gamma;\reals) \suchthat \cbdy\!f \in C^2_b(\reals)\}
$$
The \emph{defect} of $f \in QH(\Gamma;\reals)$ is the non-negative real number 
$$\sup \abs{\cbdy\!f(g_0,g_1)} = \abs{f(g_0) - f(g_0g_1) + f(g_1)}
$$ 
which measures how far $f$ deviates from being a homomorphism. The quotient space $\wt{QH}(\Gamma,\reals)$ is defined by modding out $QH(\Gamma;\reals)$ by the subspace spanned by those $f \from \Gamma \mapsto \reals$ which are either actual homomorphisms or bounded functions. One obtains an exact sequence 
\begin{align*}
0 \to \wt{QH}(\Gamma;\reals) \to & H^2_b(\Gamma;\reals) \to H^2(\Gamma;\reals) \\
\intertext{by applying the snake lemma to the short exact sequence of cochain complexes of the form}
1 \mapsto & C^*_b(\Gamma;\reals) \to C^*(\Gamma;\reals) \to C^*(\Gamma;\reals) / C^*_b(\Gamma;\reals) \to 1
\end{align*}
Thus to prove that $H^2_b(\Gamma;\reals)$ contains an embedded $\ell^1$, it is sufficient to prove the same for $\wt{QH}(\Gamma;\reals)$.

\subsection{Variant definitions of WWPD} 
\label{SectionWPDandWWPD}
In this section we fix a hyperbolic action $\Gamma \act X$. All of the definitions in this section are relative to the given action, which often goes without mention if the context is clear. After quickly reviewing concepts of ``weak proper discontinuity'' or WPD \cite{\BeFujiTag,Osin:AcylHyp}, we study in detail the concepts of ``\emph{really} weak proper discontinuity'' or WWPD \cite{BBF:SCLonMCG}.

A classification of hyperbolic actions, based on the dynamics of the extended actions on $\overline X$, may be found in Gromov's original work \cite{Gromov:hyperbolic}. We shall focus on three special classes of actions, the only ones that are relevant to our current setting (see also Corollary~\ref{CorollaryActionClasses} and the preceding paragraph). Following \BeFuji, the action $\Gamma \act X$ is \emph{elementary} if there exists a pair of loxodromic elements $\gamma,\delta \in \Gamma$ which are \emph{independent}, meaning that $\bdy\gamma \intersect \bdy\delta = \emptyset$. Also, the action is \emph{axial} if there exists a loxodromic element $\gamma \in \Gamma$ such that the two-point subset $\bdy \gamma \subset \bdy X$ is $\Gamma$-invariant. Finally, the action is \emph{elliptic} if every element is elliptic. 

%

In \BeFuji\ a ``global'' WPD property was defined by Bestvina and Fujiwara as a property of the whole group action $\Gamma \act X$. In his work on acylindrically hyperbolic groups \cite{Osin:AcylHyp}, Osin realized, and exploited, the fact that the global WPD property is expressed as a universal quantification over the following ``individual'' WPD property for loxodromic elements~$h \in \Gamma$:

\begin{definition}[\cite{Osin:AcylHyp}]
\label{DefIndividualWPD}
Given a loxodromic element $h \in \Gamma$, we say that \emph{$h$ satisfies WPD with respect to~$\Gamma$} if for every $x \in X$ and $R > 0$ there exists an integer $M>0$ such that any subset $Z \subset \Gamma$ satisfying the following property is finite: 
\begin{itemize}
\item For each $g \in Z$ we have $d(x,g(x)) < R$ and $d(h^M(x),gh^M(x)) < R$.
\end{itemize}
\end{definition}

\begin{definition}[\BeFuji]
\label{DefWPD}
The action $\Gamma \act X$ is said to satisfy WPD if it is nonelementary and every loxodromic element of $\Gamma$ satisfies~WPD with respect to~$\Gamma$. 
\end{definition}

\medskip

A \emph{really} weak proper discontinuity property WWPD, for individual loxodromic elements of a hyperbolic action, was formulated by Bestvina, Bromberg, and Fujiwara in \cite{BBF:MCGquasitrees}. We use the formulation from their paper \cite{BBF:SCLonMCG}, together with several other formulations whose equivalence is proved in Proposition~\ref{PropDefWWPD} below. One may think of WWPD as a fragment of the individual WPD property that is designed to allow for looser algebraic behavior while still capturing the essential dynamic behavior. We give four equivalent versions of WWPD, all sharing the following notation: 
\begin{itemize}
\item $\Gamma \act X$ is a group action on a hyperbolic complex;
\item $h \in \Gamma$ is a loxodromic element of the action.
\end{itemize}
The first version of WWPD is the fragment of the individual WPD property (Definition~\ref{DefIndividualWPD}) where instead of counting the total number of group elements that move point pairs a small amount, one counts only the number of left cosets of $\Stab(\bdy_\pm h)$ represented by such group elements.


\vspace{-3mm}
\paragraph{WWPD~\pref{ItemWWPDAcyl}: $h$ is equivariantly acylindrical:} \quad
\vspace{-2mm}
\begin{enumerate}
\item\label{ItemWWPDAcyl}
For every $x \in X$ and $R>0$ there exists an integer $M \ge 1$ such that any subset $Z \subset \Gamma$ that satisfies the following two properties is finite:
\begin{enumerate}
\item\label{ItemAcylindrical}
For each $g \in Z$ we have $d(x,g(x)) < R$ and $d(h^M(x),g h^M(x)) < R$; and
\item\label{ItemDifferentCosets}
No two elements of $Z$ lie in the same left coset of $\Stab(\bdy_\pm h)$.
\end{enumerate}
\end{enumerate} 
The second version of WWPD was stated in the introduction, and focuses on discreteness at infinity. It is shown in the proof of \BeFuji\ Proposition 6 that if $h \in \Gamma$ satisfies the individual WPD property then its endpoint pair $\bdy_\pm h \in \bdy X \times \bdy X - \Delta$ has discrete orbit under the diagonal action of $\Gamma$ on $\bdy X \times \bdy X - \Delta$ and the stabilizer subgroup of the endpoint pair $\bdy_\pm h$ is virtually infinite cyclic. The converse also holds, as we make explicit in Corollary~\ref{CorollaryWPDDiscrete} below. As it turns out, when generalizing from WPD to WWPD, what one loses control over is the structure of the group $\Stab(\bdy_\pm h)$ (and of its index~$\le 2$ supergroup $\Stab(\bdy h)$); what one retains is precisely the discrete topology of the orbit of $\bdy_\pm h$. 

\vspace{-3mm}
\paragraph{WWPD~\pref{ItemEqDiscYes}: $h$ has equivariantly discrete endpoint pairs:} \quad
\vspace{-2mm}
\begin{enumeratecontinue}
\item\label{ItemEqDiscYes}
The $\Gamma$-orbit of $\bdy h$ in the space $\bdy^2 X$ is a discrete subset, equivalently the $\Gamma$-orbit of $\bdy_\pm h$ in the space $\bdy X \times \bdy X - \Delta$ is a discrete subset.
\end{enumeratecontinue}
Note that equivalence of the two versions of WWPD~\pref{ItemEqDiscYes} is easily deduced using the $\Gamma$-equivariant degree 2 covering map $\bdy X \times \bdy X - \Delta \mapsto \bdy^2 X$, the deck group of which swaps the two Cartesian factors of $\bdy X \times \bdy X$.

The third version of WWPD is from \cite{BBF:SCLonMCG}:

\vspace{-3mm}
\paragraph{WWPD~\pref{ItemWWPDQuasiAxes}: $h$ has equivariantly bounded projections} \quad
\vspace{-2mm}
\begin{enumeratecontinue}
\item\label{ItemWWPDQuasiAxes} For any quasi-axis $\ell$ of $h$ there exists $D \ge 0$ such that for any $g \in \Gamma$, if $g \not\in \Stab(\bdy h)$ then the image of a closest point map $g(\ell) \mapsto \ell$ has diameter~$\le D$.
\end{enumeratecontinue}
The fourth and final version of WWPD is a variant of WWPD~\pref{ItemWWPDAcyl}, but with somewhat different logical structure. We state WWPD~\pref{ItemBadRayNo} in two forms, where a quantifier is switched. 

\vspace{-3mm}
\paragraph{WWPD~\pref{ItemBadRayNo}: $h$ is equivariantly acylindrical (variant):} \quad
\vspace{-2mm}
\begin{enumeratecontinue} 
\item\label{ItemBadRayNo}
In the group $\Gamma$ \emph{there is NO} infinite sequence $g_1,g_2,g_3,\ldots$ satisfying the following properties:
\vspace{-2mm}
\begin{enumerate}
\item\label{ItemBadRayDiffCosets}
For all $i \ne j$ the elements $g_i,g_j$ lie in different left cosets of $\Stab(\bdy_\pm h)$.
\item\label{ItemBadRayUnmoved}
For all (There exists) $x \in X$, there exists $R>0$ such that for all \break $M \ge 0$ there exists $I \ge 0$ such that if $0 \le m \le M$ and if $i \ge I$ then \break $d(g_i \, h^m(x), h^m(x)) < R$.
\end{enumerate}
\end{enumeratecontinue}
The two forms of WWPD~\pref{ItemBadRayNo} are: 

$\bullet$~\pref{ItemBadRayNo}${}_\forall$ with ``\emph{For all} $x \in X$'' in~\pref{ItemBadRayUnmoved}; 

$\bullet$~\pref{ItemBadRayNo}${}_\exists$ with ``\emph{There exists} $x \in X$'' in~\pref{ItemBadRayUnmoved}. 

\noindent
Property WWPD~\pref{ItemBadRayNo}${}_\forall$ is applied in \cite[Section~5.2]{HandelMosher:BddCohomologyI} when setting up the proof of the ``WWPD Construction Theorem'', the main technical result of that paper.


The following proposition asserts the equivalence of the various forms of WWPD, and states some further properties of WWPD elements.

\begin{PropAndDef}[Characterization and Definition of WWPD]
\label{PropDefWWPD}
Given a loxodromic element $h \in \Gamma$ of the hyperbolic action $\Gamma \act X$, WWPD Properties \pref{ItemWWPDAcyl}, \pref{ItemEqDiscYes}, \pref{ItemWWPDQuasiAxes}, \pref{ItemBadRayNo}${}_\forall$, and \pref{ItemBadRayNo}$_{\exists}$ are all equivalent to each other. We define $h$ to be WWPD (with respect to the action $\Gamma \act X$) if $h$ is loxodromic and satisfies those properties.

In addition, if $h$ satisfies WWPD then the following hold:
\begin{enumeratecontinue}
\item\label{ItemWWPDNiceEnds}
$\Stab(\bdy_- h) = \Stab(\bdy_+ h) = \Stab(\bdy_\pm h)$. In particular, for every loxodromic element $\gamma \in \Gamma$, the set $\bdy \gamma$ is either equal to or disjoint from $\bdy h$. 
\item\label{ItemWWPDNoStabInd}
If $k \in \Gamma - \Stab(\bdy h)$ then $h$ and $khk^\inv$ are independent loxodromic elements.
\end{enumeratecontinue}
\end{PropAndDef}

Before proving Proposition~\ref{PropDefWWPD}, we first derive some corollaries:

\begin{corollary} \label{CorollaryWPDDiscrete} 
A loxodromic $h \in \Gamma$ satisfies the individual WPD property if and only if $h$ satisfies WWPD and the subgroup $\Stab(\partial_\pm h) \subgroup \Gamma$ is virtually cyclic.
\end{corollary}

\begin{proof} Assuming $h$ satisfies individual WPD, clearly $h$ satisfies WWPD~\pref{ItemWWPDAcyl}. Also, the proof of \BeFuji\ Proposition~6 --- which uses only the individual WPD property for $h$ --- shows that $\Stab(\partial_\pm h) \subgroup \Gamma$ is virtually cyclic.

Suppose that $h$ satisfies WWPD~\pref{ItemWWPDAcyl} and that $\Stab(\bdy_\pm h)$ is virtually cyclic, and so the cyclic subgroup $\<h\> < \Stab(\bdy_\pm h)$ has finite index. Fixing $x \in X$ and $R>0$, consider for each integer $N>0$ the following set:
$$Z_N = \{g \in \Gamma \suchthat d(x,g(x)) < R \,\,\text{and}\,\, d(h^N(x),g h^N(x)) < R\}
$$ 
By WWPD~\pref{ItemWWPDAcyl}, there exists $N$ such that $Z_N$ intersects only finitely many left cosets of the subgroup $\Stab(\bdy_\pm h)$. Since $\Stab(\bdy_\pm h)$ is virtually cyclic, the cyclic subgroup $\<h\> < \Stab(\bdy_\pm h)$ has finite index, hence $Z_N$ intersects only finitely many left cosets of $\<h\>$. Since $h$ is loxodromic, each of its left cosets has finite intersection with $Z_N$. The set $Z_N$ is therefore finite, showing that $h$ satisfies WPD.
\end{proof}


The next corollary shows that, unlike WPD, the WWPD property behaves well with respect to pullback actions:

\begin{corollary}
\label{CorollaryWPDPullbackIsWWPD}
Consider an isometric group action $\A \from \Gamma \act X$ on a hyperbolic space with image subgroup $Q = \image(\A) \subgroup \Isom(X)$. For each $h \in \Gamma$ having image $q \in Q$, the element $h$ is a loxodromic WWPD element of the action $\Gamma \act X$ if and only if $q$ is a loxodromic WWPD element of the action $Q \act X$. 
\end{corollary} 


\begin{proof} Since $q$ and $h$ induce the same homeomorphism of $\overline X$, it follows that $q$ is loxodromic if and only if $h$ is loxodromic. In that case $\bdy_\pm h = \bdy_\pm q$, and the orbits $H \cdot \bdy_\pm h$ and $Q \cdot \bdy_\pm q$ are the same subset of $\bdy X \times \bdy X - \Delta$. Applying WWPD~\pref{ItemEqDiscYes} using $q$ it follows first that the subset $H \cdot \bdy_\pm h = Q \cdot \bdy_\pm q$ is discrete in $\bdy X \times \bdy X - \Delta$, and it then follows that $h$ also satisfies WWPD~\pref{ItemEqDiscYes}.
\end{proof}

Hyperbolic group actions are classified by their dynamical behavior at infinity \cite[Section 8.3]{Gromov:hyperbolic} (see also \cite[Section 3]{Osin:AcylHyp} and \cite{Hamann:GroupActions}). There are five classes, two of which are \pref{ItemWWPDNonelementary} and~\pref{ItemWWPDAxial} in the following corollary. Two others have no loxodromic elements: an elliptic action has solely elliptic elements; and a parabolic action is not elliptic and has no loxodromic element, but it does have a point $\xi \in \bdy X$ that is fixed by each element of $\Gamma$. The fifth class of ``quasiparabolic actions'' will be ruled out in the following proof.

\begin{corollary}
\label{CorollaryActionClasses}
If the hyperbolic action $\Gamma \act X$ possesses a loxodromic WWPD element then one of the following holds:
\begin{enumerate}
\item\label{ItemWWPDNonelementary} The action is nonelementary, meaning that there exists an independent pair of loxodromic elements; or
\item\label{ItemWWPDAxial} The action is axial, meaning that there exist $\xi \ne \eta \in \bdy X$ such that each element of~$\Gamma$ preserves the set $\{\xi,\eta\}$.
\end{enumerate} 
\end{corollary}

\begin{proof} It remains only to rule out that the action $\Gamma \act X$ is \emph{quasiparabolic} \cite[Section 8.3]{Gromov:hyperbolic}, defined by existence of $\xi \in \bdy X$ fixed by every element of $\Gamma$, and existence of loxodromic elements $g_1,g_2 \in \Gamma$ such that $\bdy_+ g_1 \intersect \bdy_+ g_2 = \{\xi\}$. Note that $\xi \in \bdy g$ for every loxodromic $g \in \Gamma$. 

Consider any loxodromic WWPD element $\gamma \in \Gamma$. By inverting $\gamma$ we may assume $\xi = \bdy_+\gamma$. Using quasiparabolicity, there exists $i \in \{1,2\}$ such that $\bdy_\pm\gamma \ne \bdy_\pm g_i$, and hence $g_i \not\in \Stab(\bdy_\pm\gamma)$, but $g_i \in \Stab(\bdy_+\gamma)$, contradicting Proposition~\ref{PropDefWWPD}~\pref{ItemWWPDNiceEnds}.
\end{proof}

For the rest of the section we turn to:

\begin{proof}[Proof of Proposition~\ref{PropDefWWPD}] First we set up some notation.

Fix a base point $O \in X$, connect $O$ to $h(O)$ by a linearly reparameterized geodesic $\ell \from [0,1] \to X$, and extend  to a continuous quasi-axis $\ell \from \reals \to X$ for $h$ satisfying $\ell(s+m) = h^m(\ell(s))$ for all $s \in \reals$ and $m \in \Z$, in particular $\ell(m)=h^m(O)$. 

We choose $k \ge 1$, $c \ge 0$ so that $\ell$ is a $k,c$-quasigeodesic, and so that any two points of $\overline X = X \union \bdy X$ are the endpoints at infinite of some $k,c$-quasigeodesic.

A choice of $k,c$ quasigeodesic in $X$ with endpoints $x,y \in \overline X = X \union \ \bdy X$ will be denoted $[x,y]$ (abusing notation in the case that $x$ or $y$ is in $\bdy X$ and hence is not a point on the quasigeodesic). For example when $g \in \Gamma$ and $a < b \in \reals$ are understood by context then, letting $x = g\ell(a)$ and $y=g\ell(b)$, the $k,c$-quasigeodesic $g \ell \restrict [a,b]$ is denoted $[x,y]$. In particular $g\ell$ decomposes as a bi-infinite concatenation of fundamental domains of the form
$$\cdots [gh^{-2}(O),gh^{-1}(O)] * [gh^{-1}(O),gh^0(O)] * [gh^0(O),gh^1(O)] * [gh^1(O),gh^2(O)] * \cdots
$$
Each of these fundamental domains is a linearly reparameterized geodesic whose diameter equals its length which equals $d(O,h(O))$. 


In contrast to the quasigeodesic notation $[x,y]$, we use $\overline{PQ}$ to denote a choice of geodesic in $X$ with endpoints $P,Q \in X$. 

\medskip

We turn now to the proof of equivalence of WWPD Properties \pref{ItemWWPDAcyl}, \pref{ItemEqDiscYes}, \pref{ItemWWPDQuasiAxes}, \pref{ItemBadRayNo}${}_\forall$, and \pref{ItemBadRayNo}$_{\exists}$.

\medskip

\textbf{Equivalence of WWPD~\pref{ItemEqDiscYes} and~\pref{ItemWWPDQuasiAxes}.} If WWPD~\pref{ItemEqDiscYes} fails then there is a sequence $g_i \in \Gamma - \Stab(\bdy_\pm h)$ such that $g_i (\bdy_\pm h) = g_i (\bdy_\pm \ell)$ converges to $\bdy_\pm h = \bdy_\pm \ell$. It follows that there are sequences $x_i,y_i \in \ell$ such that in $\overline X$ we have $x_i \to \bdy_- h$, $y_i \to \bdy_+ h$, and $x_i,y_i$ are within uniformly bounded distance of $g_i(\ell)$. The image of any closest point map $\pi_i \from g_i(\ell) \mapsto \ell$ therefore has points within uniformly bounded distance of $x_i,y_i$, and since $d(x_i,y_i)\to+\infinity$ it follows that $\text{diam}(\image(\pi_i)) \to +\infinity$, and so WWPD~\pref{ItemWWPDQuasiAxes} fails for the quasi-axis~$\ell$.

Conversely suppose that WWPD~\pref{ItemWWPDQuasiAxes} fails for some quasi-axis of $h$.  Since any two quasi-axes of $h$ have finite Hausdorff distance, it easily follows that WWPD~\pref{ItemWWPDQuasiAxes} fails for every quasi-axis of~$h$, in particular it fails for the quasi-axis~$\ell$. Thus there is a sequence $g_i \in \Gamma - \Stab(\bdy h)$ such that, letting $\pi_i \from g_i(\ell) \to \ell$ be a closest point map, the diameter of $\image(\pi_i)$ goes to $+\infinity$ with~$i$. After postcomposing each $g_i$ with an appropriate power of $h$, there is a subsegment $[x_i,y_i]$ of $\ell$ such that $x_i \to \bdy_- h$ and $y_i \to \bdy_+ h$ and such that $\image(\pi_i)$ contains a subsegment that is uniformly Hausdorff close to $[x_i,y_i]$. It follows that the sequence of unordered pairs $g_i(\bdy h) \ne \bdy h$ limits to the unordered pair $\bdy h$ in $\bdy^2 X$, proving that WWPD~\pref{ItemEqDiscYes} fails.

\medskip\textbf{Equivalence of WWPD~\pref{ItemBadRayNo}${}_\forall$  and~\pref{ItemBadRayNo}${}_\exists$.} 
To prove that implication \pref{ItemBadRayNo}${}_{\exists} \implies$~\pref{ItemBadRayNo}${}_{\forall}$ holds, suppose that \pref{ItemBadRayNo}${}_{\forall}$ is false, so there is an infinite sequence $g_1,g_2,g_3,\ldots \in \Gamma$ that satisfies \pref{ItemBadRayDiffCosets}, and that satisfies \pref{ItemBadRayUnmoved} \emph{for all} $x \in X$, hence it satisfies \pref{ItemBadRayUnmoved} \emph{for some} $x \in X$. 

For the converse implication \pref{ItemBadRayNo}${}_{\forall} \implies$~\pref{ItemBadRayNo}${}_{\exists}$, suppose that \pref{ItemBadRayNo}${}_{\exists}$ is false, and so there is an infinite sequence $g_1,g_2,g_3,\ldots \in \Gamma$ that satisfies \pref{ItemBadRayDiffCosets} and that satisfies \pref{ItemBadRayUnmoved} \emph{for some} $x' \in X$. Thus there exists $R' > 0$ such that for all $M \ge 0$ there exists $I \ge 0$ such that if $0 \le m \le M$ and if $i \ge I$ then 
$$d(g_i h^m(x'),h^m(x')) < R'
$$ 
For any $x \in X$, let $R = R' + 2 d(x,x')$. For any $M \ge 0$, choosing $I \ge 0$ so that $0 \le m \le M$ and $i \ge I$ together imply that the previous inequality holds, it follows~that
\begin{align*}
d(g_i h^m(x),h^m(x)) &\le d(g_i h^m(x),g_i h^m(x')) + d(g_i h^m(x'),h^m(x')) + d(h^m(x'),h^m(x)) \\
  &\le d(x,x') + R' + d(x',x) \\
  &\le R
\end{align*}
This shows that \pref{ItemBadRayUnmoved} is satisfied \emph{for all} $x \in X$, hence  \pref{ItemBadRayNo}${}_{\forall}$ is false.

\medskip

It remains to prove the chain of implications \pref{ItemBadRayNo}${}_{\exists} \implies$\pref{ItemWWPDAcyl} $\implies$ \pref{ItemEqDiscYes} $\implies$ \pref{ItemBadRayNo}${}_{\exists}$. 


\medskip

\textbf{WWPD~\pref{ItemBadRayNo}${}_\exists \!\!\implies$\!WWPD~\pref{ItemWWPDAcyl}.} Suppose that $h$ fails to satisfy WWPD~\pref{ItemWWPDAcyl}, so there exists $x \in X$ and $R'>0$, and for every $M \ge 1$ there exists an infinite subset $Z_M \subset \Gamma$ satisfying: (a) for each $g \in Z_M$ we have $d(x,g(x)) < R'$ and $d(h^M(x),g h^M(x)) < R'$; and (b) no two elements of $Z_M$ lie in the same left coset of $\Stab(\bdy_\pm h)$. 
%

Since each $g \in Z_M$ moves the points $x=h^0(x)$ and $h^{M}(x)$ a distance at most $R'$, by an elementary argument in Gromov hyperbolic geometry there exists a constant $R$ depending only on $R'$, $k$, $c$, and $\delta$ such that all points on the quasi-axis between $x=h^0(x)$ and $h^{M}(x)$ are moved by each $g \in Z_M$ a distance at most~$R$. In particular, if $g \in Z_M$ and if $0 \le m \le M$ then $d(g h^m(x),h^m(x)) < R$. Construct a sequence $(g_i)$ by induction so that $g_i \in Z_i$ and so that no two of $g_1,\ldots,g_i$ are in the same left coset of $\Stab(\bdy_\pm h)$. For all $M \ge 0$, letting $I = M$, it follows that if $0 \le m \le M$, and if $i \ge I=M$, then $d(g_i h^m(x),h^m(x)) < R$, proving that $h$ fails to satisfy WWPD~\pref{ItemBadRayNo}${}_\exists$.


\medskip

\textbf{WWPD~\pref{ItemEqDiscYes}$\implies$WWPD~\pref{ItemBadRayNo}${}_\exists$.} Assuming that WWPD~\pref{ItemBadRayNo}${}_\exists$ fails, we may pick an infinite sequence of elements $f_1,f_2,f_3,\ldots$ lying in different left cosets of $\Stab(\bdy_\pm h)$, and we may pick $x \in X$ and $R > 0$, so that the following holds: for all $M \ge 0$ there exists $I \ge 0$ such that if $0 \le m \le M$ and if $i \ge I$ then $d(f_i h^m(x),h^m(x)) < R$.
%
%
%
It follows that there exists a growing sequence of positive even integers $m_i=2k_i \to +\infinity$ such that for each $i$ the isometry $f_i$ moves each of the points $x, h(x),h^2(x)\ldots,h^{2k_i}(x)$ a distance at most $R$. Let $g_i = h^{-k_i} f_i h^{k_i}$, and note that $g_i$ moves each of the points $h^{-k_i}(x), h^{-k_i+1}(x),\ldots,h^{k_i-1}(x), h^{k_i}(x)$ a distance at most~$R$. Since the sequence $k \mapsto h^k(x)$ defines a quasi-isometric embedding $\Z \mapsto X$, it follows that the sequence $(h^{-k_i}(x),h^{k_i}(x))$ converges to $\bdy_\pm h$ in the space of ordered pairs of points in $X \union \bdy X$, so the sequence $(g_i(h^{-k_i}(x)),g_i(h^{k_i}(x)))$ converges to $\bdy_\pm h$, and therefore the sequence $\bdy_\pm (g_i h g^\inv_i) = g_i(\bdy_\pm h)$ converges to $\bdy_\pm h$. We may assume, dropping at most a single term of the sequence, that $f_i \not\in \Stab(\bdy_\pm h)$ for each~$i$, and hence $g_i \not\in \Stab(\bdy_\pm h)$. Since $g_i(\bdy_\pm h) \ne \bdy_\pm h$ but $g_i(\bdy_\pm h)$ converges to $\bdy_\pm h$, we have proved that $h$ does not have equivariantly discrete fixed set and so WWPD~\pref{ItemEqDiscYes} fails.

\medskip\textbf{WWPD~\pref{ItemWWPDAcyl}$\implies$WWPD~\pref{ItemEqDiscYes}.} 
This implication has a longer proof than any of the previous ones.
The main idea is that although failure of WWPD~\pref{ItemEqDiscYes} easily guarantees that an arbitrarily long segment of the quasi-axis $\ell$ is Hausdorff close to a segment of some translate $g\ell$, in order to prove failure of WWPD~\pref{ItemWWPDAcyl} one needs those two segments to be synchronized, in a sense that we formalize as the ``Synchronization Property'' below. After stating and proving that property, we will apply it to prove the implication WWPD~\pref{ItemWWPDAcyl}$\implies$WWPD~\pref{ItemEqDiscYes}.

Let $L=L(k,c,\delta) \ge 0$ be a thinness constant for $k,c$ quasigeodesic quadrilaterals in the $\delta$-hyperbolic space $X$, with finite or infinite endpoints. What this means is that for any $w,x,y,z \in \overline X$ and any $k,c$ quasigeodesics $[w,x]$, $[x,y]$, $[y,z]$, $[z,w]$ in $X$, the quasigeodesic $[w,x]$ is contained in the $L$-neighborhood of the union of the other three quasigeodesics 
$$[x,y] \union [y,z] \union [z,w]
$$

Note that if $P \in X$ has distance $\le D$ from some point on $g\ell[i,j] = [gh^i(O),gh^j(O)]$ for some integers $i \le j$, then $P$ has distance $\le D+d(O,h(O))$ from $g\ell(n)$ for some integer $n \in [i,j]$, because each fundamental domain of $g\ell$ is a geodesic segment of the form $[gh^{m-1}(O),gh^m(O)]$ having length equal to $d(O,h(O))$. For the particular case where $D=L$, we set 
$$L_1=L+d(O,h(O))
$$

\begin{description}
\item[Synchronization Property:] There exists a constant $L_2  = L_2(k,c,\delta) > 0$ such that for all $g \in \Gamma$ and all integers $M, N \ge 0$, if $d(O,g O) \le L_1$, and if the asynchronous bound $d(h^M(O),g h^N(O)) \le L_1$ holds, then for all integers \hbox{$0 \le m \le \max\{M,N\}$} the synchronous bound $d(h^m(O),g h^m(O)) \le L_2$ holds.
\end{description}
For the proof we may assume $M \ge N$, because the opposite case where $M \le N$ can be reduced to the case $M \ge N$ by replacing $g$ with $g^\inv$. 

Consider the quadrilateral having the $k,c$-quasigeodesic 
$$\ell \restrict [0,M]= [O,h^M(O)] = [O,h^m(O)] * [h^m(O),h^M(O)]
$$ 
on one side, geodesics $\overline{O,gO}$ and $\overline{gh^N(O) , h^M(O)}$ on the two adjacent sides, and the $k,c$-quasigeodesic $g\ell \restrict [0,N] = [gO,gh^N(O)]$ on the opposite side. The point $h^m(O)$ is contained in the $L$ neighborhood of the union 
$$\overline{O,gO} \union [gO,gh^N(O)] \union \overline{gh^N(O) , h^M(O)}
$$
Since each of the two geodesics $\overline{O,gO}$ and $\overline{gh^N(O) h^M(O)}$ have length $\le L_1$ it follows that $h^m(O)$ is contained with the $L+L_1$ neighborhood of some point of $[gO,gh^N(O)]$, hence there exists an integer $n \in [0,N]$ such that 
$$d(h^m(O),g h^n(O)) = d(h^m(O),g\ell(n)) \le L + L_1 + d(O,h(O)) = 2L_1
$$

Next we use this inequality to find an upper bound to $\abs{m-n}$. For this purpose we may again assume we are in the case $m \ge n$, because the opposite case can be reduced to this one by replacing $g$ with $g^\inv$. Denote
$$P = h^n(O), \,\,\,\,  Q = h^m(O), \,\,\,\, R = g(O), \,\,\,\,  S = g h^n(O)
$$
Consider the quadrilateral having on one side the $k,c$-quasigeodesic 
$$[O,Q] = \ell \restrict [0,m] = \bigl(\ell \restrict [0,n]\bigr) * \bigl(\ell \restrict [n,m]\bigr) = [O,P] * [P,Q]
$$
and having geodesics $\overline{OR}$, $\overline{RS}$, $\overline{SQ}$ on the other three sides. It follows that $P$ has distance $\le L$ from the union $\overline{OR} \union \overline{RS} \union \overline{SQ}$. Since $\Length(\overline{OR})=d(O,R) \le L_1$ and $\Length(\overline{SQ})=d(S,Q) \le 2L_1$, it follows that there is a point $U \in \overline{RS}$ such that $d(P,U) \le L + 2L_1$.
We now have
\begin{align*}
d(R,U) + d(U,S) & = d(R,S) = d(O,P) \\ &\le d(O,R) + d(R,U) + d(U,P) \\
d(U,S) &\le d(O,R) + d(U,P) \\
  &\le L_1 + (L + 2L_1) = L + 3L_1\\
d(P,Q) &\le d(P,U) + d(U,S) + d(S,Q) \\
  &\le  (L+2L_1) + (L + 3L_1) + 2L_1 = 2L + 7L_1\\
\abs{m-n}  &\le k \, d(\ell(m),\ell(n)) + kc = k \, d(Q,P) + kc \\
  &\le  k(2L+7L_1)+kc
\end{align*}
and so
\begin{align*}
d(h^m(O),g h^m(O)) &\le d(h^m(O),gh^n(O)) + d(gh^n(O),gh^m(O)) \\
&\le 2L_1 + d(g\ell(n),g\ell(m)) \\
&\le 2L_1 + k \abs{m-n} + c \\
&\le 2L_1 + k \bigl(k(2L+7L_1)+kc\bigr)  + c = L_2
\end{align*}

\bigskip

We turn now to the proof that WWPD~\pref{ItemWWPDAcyl}$\implies$WWPD~\pref{ItemEqDiscYes}. Assuming that $WWPD~\pref{ItemEqDiscYes}$ fails, there is a sequence $(f_i)_{i \ge 0}$ in the group $\Gamma$ such that $\bdy_\pm h \ne f_i(\bdy_\pm h) = \bdy_\pm(f_i h f_i^\inv)$ and such that $f_i(\bdy_- h) \to \bdy_- h$ and $f_i(\bdy_+ h) \to \bdy_+ h$ as $i \to +\infinity$. Notice that for some values of $i$ we might have $f_i(\bdy_- h) = \bdy_- h$ and for others $f_i(\bdy_+ h)=\bdy_+ h$, although those cannot happen for the same value of~$i$. 

After extracting a subsequence of $f_i$ we may assume that the following three things hold. First, if $i \ne j$ then $f_i(\bdy_\pm h) \ne f_j(\bdy_\pm h)$, and hence $f_i^\inv f_j \not\in \Stab(\bdy_\pm h)$. Second (and after possibly inverting $h$) the point $\bdy_+(f_i h f_i^\inv) = f_i(\bdy_+ h)$ is distinct from both of the points $\bdy_+ h$ and $\bdy_- h$. Third, there are integer sequences $a_i \to -\infinity$, $b_i \to +\infinity$, and $c_i < d_i$ such that 
\begin{equation}\label{EquationLPlusDelta}
d(\ell(a_i ),f_i\ell(c_i )) \le L_1
\end{equation} 
and 
\begin{equation}\label{EquationLOneAlone}
d(\ell(b_i ),f_i\ell(d_i )) \le L_1
\end{equation}
The existence of these sequences is a consequence of the thinness property for $k,c$-quasigeodesic quadrilaterals, applied to the sequence of quadrilaterals having $\ell$ on one side with ideal endpoints $\bdy_- h$, $\bdy_+ h$, having $f_i\ell$ on the opposite side with ideal endpoints $f_i(\bdy_- h)$, $f_i(\bdy_+h)$ on the opposite side, and having as its other two sides a $k,c$-quasigeodesic with ideal endpoints $\bdy_- h, f_i \bdy_- h$ and another with ideal endpoints $\bdy_+ h, f_i \bdy_+ h$. 

Let $g_i = h^{-a_i} f_i h^{c_i}$. We next prove, after passing to a subsequence, that for any $i \ne j \ge 0$ the elements $g_i,g_j$ lie in different left cosets of $\Stab(\bdy_\pm h)$, and so the set $\{g_i\}_{i \ge 0}$ is infinite and it satisfies WWPD~\pref{ItemDifferentCosets}. Fixing $j$ it suffices to prove that for all sufficiently large $i$ we have $g_i^\inv g_j \not\in \Stab(\bdy_+ h)$. Consider those values of $i$ for which this fails to hold, equivalently:
\begin{align*}
g_i^\inv g_j = h^{-c_i} f_i^\inv h^{a_i-a_j} f_j h^{c_j} &\in \Stab(\bdy_+ h) \\
f_i^\inv h^{a_i-a_j} f_j &\in \Stab(\bdy_+ h) \\
h^{a_i-a_j}(f_j(\bdy_+ h)) &= f_i(\bdy_+ h)
\end{align*}
On the left hand side of the last equation,  $f_j(\bdy_+ h) \ne \bdy_+ h$, and since $a_i - a_j \to -\infinity$ as $i \to +\infinity$, it follows by source-sink dynamics that $h^{a_i-a_j}(f_j(\bdy_+ h)) \to \bdy_- h$. But on the right hand side we have $f_i(\bdy_+ h) \to \bdy_+ h$. 
The last equation is therefore impossible for all sufficiently large~$i$.

To finish the proof that $h$ fails to satisfy WWPD~\pref{ItemWWPDAcyl}, letting $x=O$, we must exhibit $R > 0$ such that for any $M \ge 1$ the set $\{g_i\}_{i \ge 0}$ has an infinite subset $Z$ each of whose elements satisfies the two inequalities of WWPD~\pref{ItemAcylindrical}, which in our current notation become:
\begin{description}
\item[WWPD~(\ref{ItemAcylindrical})(i):] $d(O,g_i(O)) < R$
\item[WWPD~(\ref{ItemAcylindrical})(ii):] $d(h^M(O),g_i h^M(O)) < R$
\end{description}
Applying Equation~\pref{EquationLPlusDelta} we have 
$$d\bigl(O,g_i(O)\bigr) = d\bigl(h^{a_i}(O), f_i h^{c_i}(O)\bigr)
   = d\bigl( \ell(a_i ), f_i\ell(c_i ) \bigr) 
   \le L_1
$$
which shows that inequality WWPD~(\ref{ItemAcylindrical})(i) holds for any $R  > L_1$. 

To pursue inequality WWPD~(\ref{ItemAcylindrical})(ii), letting $B_i = b_i - a_i$ and $D_i = d_i - c_i$, and applying equation~\pref{EquationLOneAlone}, we have
\begin{align*}
d(h^{B_i}(O),g_i h^{D_i}(O)) &= d\bigl( h^{b_i-a_i}(O), h^{-a_i} f_i h^{c_i}(h^{d_i-c_i}(O)) \bigr)    \\
   &= d\bigl( h^{b_i}(O),f_i h^{d_i}(O) \bigr) \\
   &= d(\ell(b_i),f_i \ell(d_i )) \\
   &\le L_1
\end{align*}
Applying the Synchronization Property it follows that for all integers $m \in \max\{B_i,D_i\}$ we have
$$d(h^m(O),g_i h^m(O)) \le L_2
$$
We know that $B_i = b_i-a_i \to +\infty$ as $i \to +\infty$. Given $M$ we can find $I$ such that if $i \ge I$ then $\max\{B_i,D_i\} \ge M$, and hence the inequality $d(h^m(O),g_i h^m(O)) \le L_2$ is true for all integers $m \in [0,M]$ and all $i \ge I$. Letting $Z = \{g_i \suchthat i \ge I\}$, it follows that WWPD~(\ref{ItemAcylindrical})(ii) holds for any $R  > L_2$. Setting $R = 1 + \max\{L_1,L_2\}$, we are done.

\medskip

This completes the proof of equivalence of WWPD versions~\pref{ItemWWPDAcyl},\! \pref{ItemEqDiscYes},\! \pref{ItemWWPDQuasiAxes},\! \pref{ItemBadRayNo}${}_\forall$ and~\pref{ItemBadRayNo}${}_\exists$.

\medskip

To prove the first sentence of \pref{ItemWWPDNiceEnds} it suffices by symmetry it suffices to prove that $\Stab(\bdy_- h) = \Stab(\bdy_\pm h)$, so given $\gamma$ such that $\gamma(\bdy_- h) = \bdy_- h$ we must prove that $\gamma(\bdy_+ h) = \bdy_+ h$. If not then $\bdy_-(\gamma h \gamma^\inv) = \gamma(\bdy_- h) =\bdy_- h$ and $\bdy_+(\gamma h \gamma^\inv) = \gamma(\bdy_+ h) \ne \bdy_+ h$. Using source sink dynamics of $h$ acting on $\bdy X$ it follows that the sequence $h^n(\bdy_\pm (\gamma h \gamma^\inv)) = \bdy_\pm( (h^n \gamma) h (h^n \gamma)^\inv)$ converges to $\bdy_\pm h$ as $n \to +\infinity$, but the terms of that sequence are all distinct from $\bdy_\pm h$, contradicting that $h$ has equivariantly discrete fixed points. 

To prove the second sentence of~\pref{ItemWWPDNiceEnds}, consider the set $\bdy\gamma = \{\bdy_-\gamma,\bdy_+\gamma\}$, and suppose that $\bdy\gamma \intersect \bdy h \ne \emptyset$. By replacing $\gamma$ and/or $h$ with their inverses we may assume that $\bdy_+ \gamma = \bdy_+ h$. From the first sentence it follows that $\gamma$ fixes the point $\bdy_- h$. However, $\bdy_-\gamma$ and $\bdy_+\gamma$ are the only two points of $\bdy X$ that are fixed by $\gamma$, and hence $\bdy_-\gamma=\bdy_- h$.

To prove \pref{ItemWWPDNoStabInd}, consider $k \in \Gamma - \Stab(\bdy h)$ and suppose that $khk^\inv$ is not independent of~$h$, so the set $\bdy h \cap k(\bdy h) = \bdy h \cap \bdy(k h k^\inv)$ consists of a single point. If $\bdy_- h = k(\bdy_- h)$ or if $\bdy_+ h = k(\bdy_+ h)$ then either of the elements $k$ or $khk^\inv$ gives an immediate contradiction to \pref{ItemWWPDNiceEnds}. If $\bdy_- h = k(\bdy_+ h) = \bdy_+(k h k^\inv) = \bdy_-(k h^\inv k^\inv)$ then $k h^\inv k^\inv$ gives a contradiction to~\pref{ItemWWPDNiceEnds}, and similarly if $\bdy_+ h = k(\bdy_- h)$.
\end{proof}

\subsection{Equivalence $g \sim h$ in a hyperbolic group action}
\label{SectionSim}

Given two paths $\gamma \from [a,b] \to X$ and $\rho \from [a',b'] \to X$ in a metric space $X$, we say that \emph{$\gamma,\rho$ are $L$-Hausdorff close rel endpoints} if $d(\gamma(a),\rho(a')) \le L$, $d(\gamma(b),\rho(b')) \le L$, and the Hausdorff distance between the images $\gamma[a,b]$ and $\rho[a',b']$ is $\le L$.

In the next definition we follow \cite[Section 1]{\BeFujiTag}:

\begin{definition}\label{DefinitionSim}
Define an equivalence relation $g \sim h$ on $\Gamma$ to mean: either $g=h$; or $g$ and $h$ are both loxodromic, and for some (any) quasi-axes $\gamma_g$ and $\gamma_h$ there exists $L$ such that for any subsegment $\gamma_g[a,b]$ there is a subsegment $\gamma_h[c,d]$ and an element $k \in \Gamma$ such that $k(\gamma_g[a,b])$ is $L$-Hausdorff close rel endpoints to $\gamma_h[c,d]$.
\end{definition}

For each $g \in \Gamma$ the orbit of $\bdy_\pm g$ under the action $\Gamma \act \bdy S \times \bdy S - \Delta$ is denoted $\Gamma \cdot \bdy_\pm g$, and its closure in $\bdy S \times \bdy S - \Delta$ is denoted $\overline{\Gamma \cdot \bdy_\pm g}$.


\begin{lemma}
\label{LemmaSimEquivProps}
Given loxodromic $g,h \in \Gamma$, the following are equivalent:
\begin{enumerate}
\item\label{LemmaSimEquivSim}
$g \sim h$
\item\label{LemmaSimEquivEq}
$\overline{\Gamma \cdot \bdy_\pm g} = \overline{\Gamma \cdot \bdy_\pm h}$
\item\label{LemmaSimEquivInt}
$\overline{\Gamma \cdot \bdy_\pm g} \intersect \overline{\Gamma \cdot \bdy_\pm h} \ne \emptyset$
\end{enumerate}
If in addition $g$ satisfies WWPD then the following condition is also equivalent:
\begin{enumeratecontinue}
\item\label{LemmaSimWWPD}
$\Gamma \cdot \bdy_\pm g = \Gamma \cdot \bdy_\pm h$.
\end{enumeratecontinue}
\end{lemma}

\begin{proof} Evidently \pref{LemmaSimEquivEq}$\implies$\pref{LemmaSimEquivInt}. Pick quasi-axes $\gamma_g$ of $g$ and $\gamma_h$ of $h$. 

To prove \pref{LemmaSimEquivSim}$\implies$\pref{LemmaSimEquivEq}, assuming $g \sim h$, there exists an $L$, and for each $n$ there exists $\alpha_n \in \Gamma$, such that $\gamma_g[-n,+n]$ is $L$-Hausdorff close rel endpoints to some subsegment of $\alpha_n(\gamma_h) = \gamma_{\alpha_n h \alpha_n^\inv}$. It follows that $\alpha_n(\bdy_\pm h) = \bdy_\pm(\alpha_n h \alpha_n^\inv) \to \bdy_\pm g$ as $n \to +\infinity$. This shows that $\bdy_\pm g \in \overline{\Gamma \cdot \bdy_\pm h}$, and so the inclusion $\overline{\Gamma \cdot \bdy_\pm g} \subset \overline{\Gamma \cdot \bdy_\pm h}$ holds. The opposite inclusion holds by symmetry.

To prove \pref{LemmaSimEquivInt}$\implies$\pref{LemmaSimEquivSim}, consider $(\xi,\eta) \in \overline{\Gamma \cdot \bdy_\pm g} \intersect \overline{\Gamma \cdot \bdy_\pm h}$. Choose sequences $\alpha_n,\beta_n$ in $\Gamma$ such that the sequences $\alpha_n (\bdy_\pm g) = \bdy_\pm(\alpha_n g \alpha_n^\inv)$ and $\beta_n (\bdy_\pm g) = \bdy_\pm(\beta_n h \beta_n^\inv)$ both approach $(\xi,\eta)$ in $\bdy X \times \bdy X - \Delta$. Choose an oriented quasigeodesic $\delta$ with initial end $\xi$ and terminal end $\eta$. By precomposing both $\alpha_n$ and $\beta_n$ with appropriate powers of $g$ we may assume that $\alpha_n(\gamma_g(0))$ and $\beta_n(\gamma_h(0))$ each stay within a uniform neighborhood of $\delta(0)$. After passing to subsequences of $\alpha_n$ and $\beta_n$ it follows there is a constant $L$ and sequences $s^-_n,t^-_n \to -\infinity$ and $s^+_n,t^+_n \to +\infinity$ such that $\alpha(\gamma_g[s^-_n,s^+_n])$ and $\beta(\gamma_h[t^-_n,t^+_n])$ are each $L/2$-Hausdorff close rel endpoints to the quasigeodesic segment $\delta[-n,+n]$, and so are $L$ Hausdorff close to each other. Since the paths $\gamma_g[s^-_n,s^+_n]$ exhaust the quasi-axis $\gamma_g$, as do the paths $\gamma_h[t^-_n,t^+_n]$, it follows that $g \sim h$.

Clearly \pref{LemmaSimWWPD}$\implies$\pref{LemmaSimEquivEq}. For the converse, assuming $g$ satisfies WWPD then by discreteness of $\Gamma \cdot \bdy_\pm g$ it follows that $\overline{\Gamma \cdot \bdy_\pm g} = \Gamma \cdot \bdy_\pm g$. Assuming \pref{LemmaSimEquivEq} we therefore have $\overline{\Gamma \cdot \bdy_\pm h} = \Gamma \cdot \bdy_\pm g$, and so $\Gamma \cdot \bdy_\pm h \subset \Gamma \cdot \bdy_\pm g$. By transitivity of the $\Gamma$-action on $\Gamma \cdot \bdy_\pm g$ it follows that $\Gamma \cdot \bdy_\pm g = \Gamma \cdot \bdy_\pm h$ which is~\pref{LemmaSimWWPD}.
\end{proof}

\textbf{Remark.} Note that if two loxodromics $g,h \in \Gamma$ satisfy $\bdy_- g = \bdy_- h$ then $g \sim h$ because $\lim_{i \to +\infty} g^i(\bdy_\pm h) = \bdy_\pm g$ and one may apply Lemma~\ref{LemmaSimEquivProps} \pref{LemmaSimEquivInt}$\implies$\pref{LemmaSimEquivSim}; similarly if $\bdy_+ g = \bdy_+ h$. On the other hand if $\bdy_- g = \bdy_+ h$, or even if $h = g^\inv$, one cannot make any general conclusions regarding whether or not $g \sim h$.

\bigskip

%
%
%
%

In \cite[Proposition 6 (5)]{\BeFujiTag} it is shown that if $\Gamma \act X$ is a nonelementary hyperbolic action that possesses a WPD element then $\Gamma$ contains a pair of inequivalent loxodromic elements $g \not\sim h$. What follows in Proposition~\ref{PropWWPDProps} is a WWPD generalization of \BeFuji\ Proposition 6, together with some details needed for later use. 

We first review some basic facts and well known facts about hyperbolic ping-pong \cite{Gromov:hyperbolic}. Recall that a hyperbolic action $F \act X$ is \emph{Schottky} if $F$ is free of finite rank and for some (any) $y \in X$ the orbit map $F \mapsto X$ given by $f \mapsto f \cdot y$ is a quasi-isometric embedding with respect to the word metric on~$F$; in particular each nontrivial element of $F$ is loxodromic. Recall also the following basic facts:
\begin{description}
\item[The Cayley tree of a Schottky group:] For any Schottky action $F \act X$, and for any free basis of $F$ with associated Cayley tree $T$, there exists an $F$-equivariant continuous quasi-isometric embedding $T \mapsto X$ which has a unique continuous extension to the Gromov bordifications $T \union \bdy T \mapsto X \union \bdy X$.
\item[Hyperbolic ping pong:] For any pair of independent loxodromic isometries $\alpha,\beta \in \Isom(X)$, and any pairwise disjoint neighborhoods $U_-$, $U_+$, $V_-$, $V_+ \subset \bdy X$ of $\bdy_-\alpha$, $\bdy_+\alpha$, $\bdy_-\beta$, $\bdy_+\beta$ respectively, there exists an integer $M \ge 1$ such that for any $m,n \ge M$ the following hold: $\alpha^m\bigl(\overline{\bdy X - U_-}\bigr) \subset U_+$; \, $\beta^n\bigl(\overline{\bdy X - V_-}\bigr) \subset V_+$; \, and $\alpha^m,\beta^n$ freely generate a Schottky action $F(\alpha^m,\beta^n) \act X$. 
\end{description}


The following proposition is a WWPD version of \cite[Proposition 6 (5)]{\BeFujiTag}.
\begin{proposition} 
\label{PropWWPDProps}
Let $\gamma \in \Gamma$ satisfy WWPD, and suppose that $\Gamma$ is nonelementary, equivalently $\Gamma \ne \Stab(\bdy \gamma)$ (Corollary~\ref{CorollaryActionClasses}). For any $\alpha \in \Gamma - \Stab(\bdy \gamma)$ there exists $A > 0$ such that for any $a \ge A$, letting $g_1=\gamma^a$ and $h_1 = \gamma^a (\alpha \gamma \alpha^\inv)^{-a} = \gamma^a \alpha \gamma^{-a} \alpha^\inv$, we have:
\begin{enumerate}
\item\label{ItemhOneLox}
$h_1$ is also loxodromic
\item\label{Item_gOne_hOneNotSim}
$g_1 \not\sim h_1$ and $g_1 \not\sim h^\inv_1$
\item \label{Item_gOne_hOneSch}
$g_1,h_1$ freely generate a Schottky subgroup of $\Gamma$.
\end{enumerate}
\end{proposition}

\begin{proof} Consider the loxodromic $\delta = \alpha \gamma \alpha^\inv$. Applying Proposition~\ref{PropDefWWPD}~\pref{ItemWWPDNoStabInd}, $\delta$ and $\gamma$ are independent. Now we play hyperbolic ping-pong with some WWPD spin. Applying WWPD~\pref{ItemEqDiscYes}, there exist pairwise disjoint neighborhoods $U_-,U_+,V_-,V_+ \subset \bdy S \times \bdy S - \Delta$ of $\bdy_- \gamma$, $\bdy_+ \gamma$, $\bdy_- \delta$, $\bdy_+ \delta$, respectively, where $U_-,U_+$ are chosen so small that $\bdy_\pm \gamma$ is the only element of the $\Gamma$-orbit of the ordered pair $\bdy_\pm \gamma$ that is contained in $U_- \times U_+$. Choose $M \ge 1$ sufficiently large so that for $m, n \ge M$ the conclusions of hyperbolic ping pong hold as stated above, in particular $\gamma^m,\delta^n$ freely generate a Schottky action on~$S$. For $a \ge A = 2M$, letting $g_1 = \gamma^a$ and $h_1=\gamma^{a}\delta^{-a} = \gamma^{a} \alpha \gamma^{-a} \alpha^\inv$, it follows that $h_1$ is loxodromic and that the two elements $g_1,h_1$ freely generate a Schottky subgroup, because any finitely generated subgroup of the Schottky group $\<\gamma^a,\delta^a\>$ is also Schottky. 

Having verified conclusions~\pref{ItemhOneLox},~\pref{Item_gOne_hOneSch} of Proposition~\ref{PropWWPDProps}, we turn to~\pref{Item_gOne_hOneNotSim}.

Letting $U'_+ = \gamma^M(U_+) \subset U_+$,
we have $\bdy_- h_1 \in V_+$, \, $\bdy_+ h_1 \in U'_+$, and $\gamma^{-M}(\bdy_\pm h_1) = \bdy_\pm(\gamma^{-M} h_1 \gamma^{M}) \in U_- \times U_+$. But $\gamma^{-M} h_1 \gamma^{M}$ is independent of $\gamma$ and so $\gamma^{-M}(\bdy_\pm h_1) \ne \bdy_\pm \gamma$. From our choice of $U_- \times U_+$ and the fact that $\gamma$ satisfies WWPD, by applying Lemma~\ref{LemmaSimEquivProps}~\pref{LemmaSimWWPD} it follows that $\gamma \not\sim \gamma^{-M} h_1 \gamma^{M}$. Since $\not\sim$ is invariant under conjugacy and under passage to powers, it follows that $g_1 \not\sim h_1$. 

The exact same argument --- with $\alpha^\inv$ in place of $\alpha$, and with the inclusion $\gamma^{-M}(U_-) \subset U_-$ in place of the inclusion $\gamma^M(U_+) \subset U_+$ --- shows that $g_1 \not\sim h'_1 = \gamma^a \alpha^\inv \gamma^{-a} \alpha$. But since $h'_1$ is conjugate to $h^\inv_1$ it follows that $g_1 \not\sim h^\inv_1$. 
\end{proof}

\paragraph{Remark.} 
Putting together Proposition~\ref{PropWWPDProps} and \BeFuji\ Theorem 1 we obtain, as an immediate corollary, the following theorem regarding group actions on hyperbolic spaces possessing WWPD elements. 

\begin{theorem}
\label{ThmWWPDHTwoBSpecial}
If $\Gamma \act S$ is a hyperbolic action possessing an independent pair of loxodromic elements, and if $\Gamma$ has a WWPD element, then $H^2_b(\Gamma;\reals)$ contains an embedded~$\ell^1$. \qed
\end{theorem}

\section{The Global WWPD Theorem} 
\label{SectionWreathProductMethods}

In \BeFuji\ Theorem 8, Bestvina and Fujiwara generalize their Theorem~7 by giving hypotheses on a hyperbolic action of a finite index subgroup $N$ of a group $\Gamma$ which are sufficient to imply the existence of an embedding of $\ell^1$ into $H^2_b(\Gamma;\reals)$. The hypotheses require that the subgroup action $N \act X$ satisfies a certain global WPD property, and that there is an embedding of $\Gamma$ into a certain wreath product which interacts with the subgroup action in a certain way. Those hypotheses are tailored to the case where $\Gamma$ is a subgroup of a surface mapping class groups, allowing one to apply Ivanov's subgroup decomposition theory to obtain a natural finite index subgroup of $\Gamma$ acting on a certain subsurface curve complex, \emph{and} allowing one to use the theorem of \BeFuji\ saying that \emph{every} element of the mapping class group that acts loxodromically on the curve complex satisfies WPD.

The Global WWPD Theorem generalizes \BeFuji\ Theorem~8 in two ways: the global WPD hypothesis of that theorem is replaced by our much weaker global WWPD hypothesis; and one does not need any wreath product hypothesis at all, instead the proof simply applies the Kaloujnin-Krasner embedding of $\Gamma$ into the wreath product of $N$ by $\Gamma/N$.

In our paper \cite{HandelMosher:BddCohomologyI}, where the Global WWPD Theorem is applied in the setting of $\Out(F_n)$, the loxodromic elements of the action we use are \emph{not} known to all satisfy WWPD, but we \emph{do} construct sufficient WWPD elements in order to verify that the Global WWPD Hypothesis holds. 

\begin{definition} 
\label{DefWWPDHypothesis} A group $\Gamma$ is said to satisfy the \emph{global WWPD hypothesis} if there exists a normal subgroup $N \normal \Gamma$ of finite index, a hyperbolic action $N \act X$, and a rank~$2$ free subgroup $F \subgroup N$, such that the following hold:
\begin{enumerate}
\item\label{ItemLoxOrEll}
Each element of $N$ acts either loxodromically or elliptically on $X$.
\item\label{ItemFirstWWPD}
The restricted action $F \act X$ is Schottky and each of its nonidentity elements is WWPD with respect to the action $N \act X$.
\item\label{ItemWWPDOrEllRestrict} For each inner automorphism $i_g \from \Gamma \to \Gamma$, letting $N \act_{g} X$ denote the composed action $N \xrightarrow{i_g} N \act X$, the restricted action $F \act_{g} X$  satisfies one of the following:
\begin{enumerate}
\item\label{ItemSchottkyRestriction} $F \act_{g} X$ is Schottky and each of its nonidentity elements is WWPD with respect to the action $N \act_{g} X$; or 
\item\label{ItemEllipticRestriction}
$F \act_{g} X$ is elliptic.
\end{enumerate}
\end{enumerate}
In situations where $N$, its action $N \act X$, and/or $F$ are specified, we shall adopt phrases like ``$\Gamma$ satisfies WWPD with respect to $N$, its action $N \act X$, and/or $F$''. 

In practice we reformulate item~\pref{ItemWWPDOrEllRestrict}, using that $N$ has finite index in order to cut down the set of inner automorphisms that must be checked to a finite subset. Since $N$ is normal in~$\Gamma$, the action of $\Gamma$ on itself by inner automorphisms restricts to an action of $\Gamma$ on $N$ that we denote $i \from \Gamma \to \Aut(N)$. Choose coset representatives $g_\kappa$ of $N$ in $\Gamma$, where $\kappa \in \{1,\ldots,K\}$ and $K = [\Gamma:N]$, and by convention choose $g_1$ to be the identity. Let $i_\kappa \in \Aut(N)$ be the restriction to $N$ of $i_{g_\kappa} \in \Inn(\Gamma)$. We refer to $i_1,\ldots,i_K$ as \emph{outer representatives} of the action of $\Gamma$ on $N$ (see diagram below). Let $N \act_\kappa X$ be the composed action $N \xrightarrow{i_\kappa} N \act X$, and let $F \act_\kappa X$ be its restriction. With this notation, to verify \pref{ItemWWPDOrEllRestrict} as stated for all the actions $F \act_{g} X$ it suffices to check only the actions $F \act_\kappa X$. In other words, \pref{ItemWWPDOrEllRestrict} is equivalent to:
\begin{itemize}
\item[$\pref{ItemWWPDOrEllRestrict}'$] For each $\kappa=1,\ldots,K$ the action $F \act_\kappa X$ satisfies one of the following:
\begin{itemize}
\item[(a)] $F \act_\kappa X$ is Schottky and each of its nonidentity elements is WWPD with respect to the action $N \act_\kappa X$; or
\item[(b)] $F \act_\kappa X$ is elliptic.
\end{itemize}
\end{itemize}
The terminology of ``outer representatives'' refers to the fact that in the commutative diagram
$$\xymatrix{
\Gamma \ar[rr]^{i} \ar[d] && \Aut(N) \ar[d] \\
\Gamma / N \ar[rr] && \Out(N) \ar@{=}[r] &\Aut(N) / \Inn(N)
}$$
the automorphisms $i_\kappa \in \Aut(N)$ represent all of the elements of the image of the homomorphism $\Gamma / N \to \Out(N)$ (that homomorphism need not be injective, and so there may be some duplication of elements of $\Out(N)$ represented by the list $i_1,\ldots,i_K$, but this is inconsequential). The equivalence of \pref{ItemWWPDOrEllRestrict} and $\pref{ItemWWPDOrEllRestrict}'$ holds because if we replace $g_\kappa$ by something else $h = \nu g_\kappa$ in its coset ($\nu \in N$), then $i_\kappa \in \Aut(N)$ is replaced by $i_h = i_\nu \composed i_\kappa \in \Aut(N)$, and so the restricted actions $F \act_\kappa X$ and $F \act_{h} X$ are conjugate by an isometry of $X$, namely the action of $\nu \in N$. But each of properties~(a) and~(b) is invariant under such conjugation. 

This completes Definition~\ref{DefWWPDHypothesis}.
\end{definition}

\emph{Remark.} We emphasize that although the various actions $N \act_\kappa X$ are equivalent up to inner automorphisms of $\Gamma$ restricted to its normal subgroup $N$, that does \emph{not} mean that they are equivalent up to conjugation by isometries of $X$, because $\Gamma$ itself does not act on~$X$. Thus, for example, an element of $N$ may be loxodromic with respect to one of the actions $N \act_\kappa X$ but not with respect to a different one.

%
%
%

Here, restated from the introduction, is our main result regarding the global WWPD hypothesis.

\begin{theoremD}[Global WWPD Theorem]
\label{ThmWWPDHTwoB}
If a group $\Gamma$ satisfies the global WWPD hypothesis then $H^2_b(\Gamma;\reals)$ contains an embedding of $\ell^1$.
\end{theoremD}

The proof of this theorem, which takes up the remainder of the paper, will hew closely to the proof of \cite{BestvinaFujiwara:bounded} Theorem~8, except for the following major difference: the proof uses the Kaloujnin--Krasner embedding to avoid the wreath product hypothesis in the statement of \cite{BestvinaFujiwara:bounded} Theorem~8. This results in many slight technical differences from \BeFuji, and so we have written the proof to be primarily self-contained, but with a few references to \cite{BestvinaFujiwara:bounded} and to~\cite{Fujiwara:H2BHyp}. We begin by setting up the notation of the Kaloujnin--Krasner embedding which is adopted throughout the proof; substantial use of algebraic properties of the embedding comes in Step~5 of the proof. 

\medskip

Fix a group $\Gamma$, a finite index normal subgroup $N \normal \Gamma$, an action $N \act X$ on a hyperbolic complex $X$, and a rank~$2$ free subgroup $F \subgroup N$, with respect to which the global WWPD hypothesis of Definition~\ref{DefWWPDHypothesis} holds: each element of $N$ is either loxodromic or elliptic; 
the restricted action $F \act X$ is a Schottky group whose nontrivial elements are all WWPD elements of the action $N \act X$; and for each inner automorphism $i \from \Gamma \to \Gamma$, the restriction to $F$ of the composed action $N \xrightarrow{i} N \act X$ is either elliptic or is a Schottky group whose nontrivial elements are all WWPD elements of the composed action. We shall prove that $H^2_b(\Gamma;\reals)$ contains an embedded $\ell^1$.

We may assume that $F \subgroup [N,N]$, because there is a natural inclusion of commutator subgroups $[F,F] \subgroup [N,N]$, and because the global WWPD hypothesis is retained when $F$ is replaced by any of its rank~$2$ subgroups, so we may replace $F$ with a rank~$2$ subgroup of $[F,F]$. This assumption will be used only at the very last sentence of the proof.

\medskip

\begin{proof}[Step 1:] \emph{The Kaloujnin--Krasner embedding.} This is an embedding of $\Gamma$ into a natural wreath product constructed from~$N \normal \Gamma$. Consider the quotient group $Q=\Gamma/N = \{B = Ng \suchthat g \in \Gamma\}$. Elements of the wreath product $N^Q \semidirect \Sym(Q)$ are denoted as ordered pairs $(\rho,\phi) \in N^Q \times \Sym(Q)$, the group operation being $(\rho,\phi) \cdot (\sigma,\psi) = (\rho \cdot (\sigma \composed \phi),\psi\phi)$; note the reversal of order in the second coordinate. The embedding $\theta \from \Gamma \inject N^Q \semidirect \Sym(Q)$ is given by the following formula.\footnote{See e.g.\ \cite{Wells:WreathProduct}, proofs of Proposition~12.1 and Theorem~11.1} For each coset $B \in \Gamma/N=Q$ choose a representative $g_B \in B$, and so $B = N g_B$; in particular, choose $g_N$ to be the identity. For each $\mu \in \Gamma$ define $\theta(\mu) = (\rho_\mu,\phi_\mu) \in N^Q \semidirect \Sym(Q)$ where for each $B \in Q$ we have:
\begin{align*}
\rho_\mu(B) &= g^\vp_B \, \mu \, g_{B\mu}^\inv \in N \\
\phi_\mu(B) &= B\mu \in Q
\end{align*}
One may easily check this to be a homomorphism. To see that it is injective, note first that $\phi_\mu$ is trivial if and only if $B=B\mu$ for all $B \in \Gamma/N$ if and only if $N\nu=N\nu\mu$ for all $\nu \in \Gamma$ if and only if $\mu \in N$ (we shall need this below). Next, if $\mu \in N$ and if $\rho_\mu(B)$ is trivial for all $B \in Q$, then using $B=N$ it follows that $\rho_\mu(N) = g^{\vphantom\inv}_N \, \mu \, g_{N\mu}^\inv = g^{\vphantom\inv}_N \, \mu \, g_N^\inv = \mu$ is trivial.

This completes Step 1.
\end{proof}

We next set up some notation that will be used for the rest of the proof of Theorem~\ref{ThmWWPDHTwoB}. In particular we match the notation regarding the Kaloujnin--Krasner embedding with our notation for outer representatives of the action of $\Gamma$ on~$N$.

Under the natural injection $N^Q \inject N^Q \semidirect \Sym(Q)$ given by $\rho\mapsto (\rho,\text{Id})$, we identify $N^Q$ with its image. Under this identification, note that $\theta(\Gamma) \intersect N^Q = \theta(N)$ because, as just seen, $\phi_\mu$ is trivial if and only if~$\mu \in N$.

Fix an enumeration of $Q = \Gamma / N$ as $N=B_1,\ldots,B_K$, and abbreviate $g_{B_\kappa}$ to~$g_\kappa$. For each $\kappa=1,\ldots,K$ consider the map $i_\kappa \from N \to N$ defined by the following composition:
$$
\xymatrix{
N \ar[r]_-{\theta} \ar@/^2pc/[rrrrr]^{i_\kappa} & \theta(N) \ar@{=}[r] & \theta(\Gamma) \intersect N^Q \ar[r]_-{\subset} & N^Q \ar[rr]_{\rho \mapsto \rho(B_\kappa)} && N
}
$$
Tracing through the definitions, and using that if $\mu \in N$ then $B\mu=B$, it follows that $i_\kappa \from N \to N$ is the outer representative of the action of $\Gamma$ on $N$ determined by $g_\kappa$: 
$$i_\kappa(\mu) = \rho_\mu(B_\kappa) = g^{\vphantom{-1}}_\kappa \, \mu \, g_{B_\kappa\mu}^\inv = g^{\vphantom{-1}}_\kappa \, \mu \, g_\kappa^\inv
$$
In particular, since $N=B_1$ is represented by $g_1=\Id$ the automorphism $i_1$ is the identity and the action $N \act_1 X$ is the given action $N \act X$; we therefore often drop the subscript $1$ from the action $N \act_1 X$.

Since each action $F \act_\kappa X$ is either Schottky or elliptic, and since $F \act_1 X$ is Schottky, after re-indexing there exists $K_1 \ge 1$ such that the action $F \act_\kappa X$ is Schottky when $1 \le \kappa \le K_1$ and it is elliptic when $K_1 < \kappa \le K$.

\medskip

Consider the equivalence relation $\sim$ on elements of the action $N \act X$, as defined in Section~\ref{SectionSim}. We let $\sim_\kappa$ denote the analogous equivalence relation on the elements of each action $N \act_\kappa X$, which is given by 
$$\mu \sim_\kappa \nu \iff i_{\kappa}(\mu) \sim i_{\kappa}(\nu)
$$

\begin{proof}[Step 2:] \emph{A good free basis for $F$.} 
We may assume, after replacing $F$ with a rank~$2$ free subgroup, that there is a free basis $F = \<g_1,h_1\>$ such that for each $\kappa=1,\ldots,K$ we have
\begin{enumerate}
\item\label{ItemBeFujiClaimOne} $g_1 \not\sim_\kappa h_1$. 
\item\label{ItemBeFujiClaimTwo} $g_1 \not\sim_\kappa h_1^\inv$. 
\end{enumerate}
To see why, start with an arbitrary free basis $F = \<g_1,h_1\>$, and suppose by induction that we have $k \in \{1,\ldots,K\}$ such that \pref{ItemBeFujiClaimOne} and~\pref{ItemBeFujiClaimTwo} hold for $1 \le \kappa < k$ (which is vacuously true when $k=1$). If $k > K_1$ then $F \act_k X$ is elliptic and so \pref{ItemBeFujiClaimOne} and~\pref{ItemBeFujiClaimTwo} hold as well for $\kappa=k$. We may therefore assume that $k \le K_1$. The subgroup $\<g_1\>$ is a maximal cyclic subgroup of $F$, since $g_1$ is a free basis element. It follows that for each of the actions $F \act_\kappa X$, $1 \le \kappa \le k$, the subgroup of $F$ that stabilizes the fixed point pair of the loxodromic isometry $g_1 \act_\kappa X$ is equal to $\<g_1\>$, and therefore $h_1$ does not stabilize the fixed point pair of $g_1$ for any of these actions. Applying Proposition~\ref{PropWWPDProps} we obtain for each $1 \le \kappa \le k$ a constant $A_\kappa$ such that if $a \ge A_\kappa$ then $g_1$ and~$h'_1 = g_1^a h_1 g_1^{-a} h_1^\inv$ satisfy properties \pref{ItemBeFujiClaimOne} and~\pref{ItemBeFujiClaimTwo}. Taking $a \ge \max\{A_\kappa \suchthat 1 \le \kappa \le K_1\}$ and replacing $h_1$ with $h'_1$, those properties hold simultaneously for all $\kappa \le K_1$. Having made this replacement, the WWPD hypothesis for $\Gamma$ with respect to $N \act X$ and $F$ remains true, completing Step~2.
\end{proof}

We shall often use the above ``maximization trick'' to obtain constant bounds independent of $\kappa$.

\subparagraph{Some notations:} Let $T$ denote the Cayley tree of the free group $F = \<g_1,h_1\>$, with a geodesic metric that assigns length~$1$ to each edge. For each $\kappa \le K_1$ we have a continuous, $K_\kappa,C_\kappa$ quasi-isometric embedding $\xi_\kappa \from T \to X$ which takes vertices to vertices and edges to edge paths and which is $F$-equivariant with respect to the action $N \act_\kappa X$. The quasi-isometry constants $K_\kappa \ge 1$, $C_\kappa \ge 0$ of the maps $\xi_\kappa$ at first depend on $\kappa$, but by taking maxima we obtain:
\begin{enumeratecontinue}
\item\label{ItemQIConstants} There exists $K \ge 1$, $C \ge 0$ independent of $\kappa$ such that $\xi_\kappa \from T \to X$ is a $K,C$ quasi-isometric embedding for $\kappa \le K_1$. 
\end{enumeratecontinue}

Fix $\delta$ to be a hyperbolicity constant for $X$. Given a nontrivial $g \in F$, let $A_T(g)$ denote its axis in the tree~$T$, along which the translation distance of $g$ denoted $L_g$ equals the cyclically reduced word length of~$g$. For each $\kappa \le K_1$, let $A_\kappa(g) = \xi_\kappa(A_T(g))$ which by item~\pref{ItemQIConstants} is a $K,C$-quasi-axis of $g$ with respect to the Schottky action \hbox{$F \act_\kappa X$}. Again we often drop the subscript $\kappa=1$, and so $A(g)=A_1(g)$ is the $K,C$-quasi-axis with respect to the given action $N \act X = N \act_1 X$. 

\begin{proof}[Step 3: A sequence of inequivalencies.] 
There exists a sequence of nontrivial elements
$$f_1,f_2,f_3,\ldots \in F \subgroup [N,N]
$$
such that the following inequivalence properties hold:
\begin{enumeratecontinue}
\item\label{ItemNotSimThree}
For each $i \ge 1$ and each $1 \le \kappa \le K$ we have $f_i \not\sim i_\kappa(f_i^\inv)$.
\item\label{ItemNotSimOne}
For each $i > j \ge 1$ and each $1 \le \kappa \le K$ we have $f_i \not\sim i_\kappa(f^\pm_j)$.
\end{enumeratecontinue}
If $\kappa > K_1$ then $i_\kappa(f_i^\inv)$ and $i_\kappa(f_j^{\pm 1})$ are both elliptic and are therefore distinct from the loxodromic element $f_i$, hence \pref{ItemNotSimThree} and \pref{ItemNotSimOne} both follow from Definition~\ref{DefinitionSim}. We may therefore assume that $\kappa \le K_1$. In that case the justifications we give will follow \BeFuji\ Proposition~2, Claim 3 together with lines of argument found in \BeFuji\ Theorem~8, on pages 81--82 (here we hew more closely to the notation of \BeFuji).

The elements $f_j$ are given by formulas
$$f_j = g_1^{-s_j} h_1^{-t_j} g_1^{m_j} h_1^{n_j} g_1^{k_j} h_1^{-l_j}
$$
with a rapidly growing sequence of integer exponents 
\begin{align*}
& m_1 \lessless n_1 \lessless k_1 \lessless l_1 \lessless s_1 \lessless t_1 \lessless \\
& m_2 \lessless n_2 \lessless k_2 \lessless l_2 \lessless s_2 \lessless t_2 \lessless \\
& m_3 \lessless n_3 \lessless \cdots
\end{align*}
which are chosen by the following protocol. There exists a sequence of functions 
\begin{align*}
M_1, \quad N_1, \quad & K_1, \quad L_1, \quad S_1, \quad T_1, \\
M_2, \quad N_2, \quad & K_2, \quad L_2, \quad S_2, \quad T_2, \\
M_3, \quad N_3, \quad &\cdots
\end{align*}
with the following properties:
\begin{itemize}
\item The $1^{\text{st}}$ function $M_1$ has $0$ arguments (and so is a constant), the $2^{\text{nd}}$ function $N_1$ has 1 argument, the $3^{\text{rd}}$ has 2 arguments, and so on, the $p^{\text{th}}$ having $p-1$ arguments.
\item If $m_1 \ge M_1$, $n_1 \ge N_1(m_1)$, $k_1 \ge K_1(m_1,n_1)$, $l_1 \ge L_1(m_1,n_1,k_1)$, and so on, the $p^{\text{th}}$ exponent being bounded below by the $p^{\text{th}}$ function applied to the previous $p-1$ exponents, then properties~\pref{ItemNotSimThree}, \pref{ItemNotSimOne} hold. 
\end{itemize}
At first these functions are chosen to depend on $\kappa \le K_1$, obtaining a sequence of functions denoted $M^\kappa_1,N^\kappa_1,\ldots,M^\kappa_2,N^\kappa_2,\ldots$ But then, by the maximization trick, one replaces the collection of constants $M^\kappa_1$ with the single constant $M_1=\max_{\kappa \le K_1}\{M^\kappa_1\}$; and one replaces the collection of 1-variable functions $N^\kappa_1(m_1)$ with the single 1-variable function $N_1(m_1)=\max_{\kappa \le K_1} \{N^\kappa_1(m_1)\}$; and so~on. 



First we prove~\pref{ItemNotSimOne}. The quasi-isometric embedding $\xi_1 \from T \to X$ restricts to a map $A_T(f_i) \xrightarrow{\xi_1} A_1(f_i)$ taking an $h_1^{-l_i}$ subsegment of $A_T(f_i)$ to a subsegment $\alpha$ of $A_1(f_i)$ (see the diagram below). The constant $l_i$ is chosen so large that if $f_i \sim i_\kappa(f_j)$ then some translate of $\alpha$ stays close rel endpoints to a subsegment of $A_\kappa(f_j)$ which contains the image under $T \xrightarrow{\xi_\kappa} X$ of an entire fundamental domain for $A_T(f_j)$ labelled $g_1^{-s_j} h_1^{-t_j} g_1^{m_j} h_1^{n_j} g_1^{k_j} h_1^{-l_j}$. Under that translate, the $g_1^{m_j}$ subsegment of this fundamental domain is taken by $\xi_\kappa$ to a subsegment $\beta$ of $A_\kappa(f_j)$, and the $h_1^{n_j}$ subsegment is taken by $\xi_\kappa$ to a subsegment $\gamma$ of $A_\kappa(f_j)$, as shown in the the following diagram. In this diagram the symbol $\mapsto_\kappa$ is an abbreviation for the map $\xi_\kappa$:

\vbox{
\centerline{$\displaystyle\cdots g_1^{-s_i} h_1^{-t_i} g_1^{m_i} h_1^{n_i} g_1^{k_i} \overbrace{h_1^\inv h_1^\inv h_1^\inv h_1^\inv h_1^\inv h_1^\inv \cdots\cdots h_1^\inv h_1^\inv h_1^\inv h_1^\inv h_1^\inv h_1^\inv}^{\text{($l_i$ repetitions) $\mapsto_1 \alpha$.}} g_1^{-s_i} h_1^{-t_i} g_1^{m_i} h_1^{n_i} g_1^{k_i} \cdots 
$}
\centerline{$\displaystyle 
 \qquad \cdots g_1^{-s_j} h_1^{-t_j} \underbrace{g_1 \, g_1 \, \cdots \, g_1 \, g_1}_{\text{($m_j$ reps.) $\mapsto_\kappa \beta$.}} \, \underbrace{h_1 \, h_1 \, \cdots \, h_1 \, h_1}_{\text{($n_j$ reps.) $\mapsto_\kappa \gamma$}} g_1^{k_j} h_1^{-l_j} \cdots
$}
}

\smallskip
\noindent
It follows that some translate of a subsegment of $\alpha$ stays close rel endpoints to $\beta\gamma$, which implies in turn that there exist integers $\xi,\eta > 0$ such that some translate of $\beta$ stays close rel endpoints to a subsegment of $\alpha$ labelled by $h_1^{-\xi}$, and some translate of $\gamma$ stays close rel endpoints to a subsegment of $\alpha$ labelled by $h_1^{-\eta}$. Since $m_j \lessless n_j$ it follows that $\xi < \eta$, and therefore some translate of $\beta$ stays close rel endpoints to a subsegment of~$\gamma$. By taking $m_j$ sufficiently large this contradicts $g_1 \not\sim_\kappa h_1$, item~\pref{ItemBeFujiClaimOne} above. Similarly if $f_i \sim i_\kappa(f_j^\inv)$, equivalently if $f_i^\inv \sim i_\kappa(f_j)$, then after inverting the first line in the above diagram, obtaining the segment $\alpha^\inv$ labelled by $l_i$ repetitions of $h_1$, one sees that some translate of a subsegment $\alpha^\inv$ stays close rel endpoints to~$\beta\gamma$, which leads similarly to a contradiction of $g_1 \not\sim_\kappa h_1$.

Next we turn to the proof of~\pref{ItemNotSimThree}. The quasi-axis $A_1(f_i)$ can be subdivided into alternating subsegments labelled $\alpha_1$ and $\beta_1$ which are the images under the map $A_T(f_i) \xrightarrow{\xi_1} A_1(f_i)$ of alternating subsegments of $A_T(f_i)$ labelled $h_1^{-t_i}$ and $g_1^{m_i} h_1^{n_i} g_1^{k_i} h_1^{-l_i} g_1^{-s_i}$. Similarly, $A_\kappa(f_i^\inv)$ can be subdivided into alternating subsegments labelled $\alpha_\kappa$ and $\beta_\kappa$, the images under $A_T(f_i^\inv) \xrightarrow{\xi_\kappa} A_\kappa(f_i^\inv)$ of alternating subsegments of $A_T(f_i^\inv)$ labelled $h_1^{t_i}$ and $g_1^{s_i} h_1^{l_i} g_1^{-k_i} h_1^{-n_i} g_1^{-m_i}$. Note that each $\alpha_1\beta_1$ subsegment of $A_1(f_i)$ is a fundamental domain, and similarly for each $\alpha_\kappa\beta_\kappa$ subsegment of $A_\kappa(f_i)$. This situation is depicted in the following diagram:
\begin{align*}
A_1(f_i): \qquad &\cdots\,\, 
\overbrace{
 h_1^{-t_i} 
}^{\alpha_1} 
\overbrace{
 g_1^{m_i} \,\, h_1^{n_i} \,\, g_1^{k_i} \,\, h_1^{-l_i} \,\, g_1^{-s_i}
}^{\beta_1} 
\overbrace{
 h_1^{-t_i} 
}^{\alpha_1} 
\overbrace{
 g_1^{m_i} \, h_1^{n_i} \, g_1^{k_i} \, h_1^{-l_i} \, g_1^{-s_i}
}^{\beta_1} 
\overbrace{
 h_1^{-t_i} 
}^{\alpha_1} 
 \,\,\cdots
\\
A_\kappa(f_i^\inv): \qquad & \cdots\,\,
\underbrace{
 h_1^{t_i} 
}_{\alpha_\kappa}
\underbrace{
 g_1^{s_i} \, h_1^{l_i} \, g_1^{-k_i} \, h_1^{-n_i} \, g_1^{-m_i}
}_{\beta_\kappa}
\underbrace{
 h_1^{t_i} 
}_{\alpha_\kappa}
\underbrace{
 g_1^{s_i} \, h_1^{l_i} \, g_1^{-k_i} \, h_1^{-n_i} \, g_1^{-m_i}
}_{\beta_\kappa}
\underbrace{
 h_1^{t_i} 
}_{\alpha_\kappa} \,\,\cdots
\end{align*}

\noindent
For sufficiently large choice of $t_i > \!\! > s_i,l_i,k_i,m_i,n_i$ it follows that $\alpha_1$ is much longer than $\beta_\kappa$ and $\alpha_\kappa$ is much longer than $\beta_1$. 

Arguing by contradiction, suppose that $f_i \sim i_\kappa(f_i^\inv)$ for arbitrarily large choices of $t_i$. Thus, there are subsegments of $A_1(f_i)$ and of $A_\kappa(f^\inv_i)$, each containing arbitrarily many fundamental domains, that are close rel endpoints to each other. This is depicted in the above diagram, but the alignment of the fundamental domains is not yet clear; by studying various cases of that alignment we shall arrive in each case at a contradiction. First, it immediately follows that $h_1 \sim i_\kappa(h_1^\inv)$, because by using that $\alpha_1$ is much longer than $\beta_\kappa$ and that $\alpha_\kappa$ is much longer than $\beta_1$ it follows that some subsegment of $\alpha_1$ labelled by a large power of $h_1^\inv$ is close after translation to some subsegment of $\alpha_\kappa$ labelled by a large power of $h_1$. 

We now consider four very similar cases of alignment. 

\textbf{Case 1:} The terminal $g_1^{-m_i/2}$ subsegment of $\beta_\kappa$ cannot be close rel endpoints to a subsegment of $\alpha_1$: otherwise, since $h_1^\inv \sim i_\kappa(h_1)$ it would follow that the terminal $g_1^{-m_i/2}$ segment of $\beta_k$ is after translation close rel endpoints to a subsegment of $\alpha_\kappa$ labelled by some large power of $h_1$, which for sufficiently large $m_i$ contradicts that $g_1^\inv \not\sim_\kappa h_1$, item~\pref{ItemBeFujiClaimTwo}.

\textbf{Case 2:} The initial $g_1^{s_i/2}$ subsegment of $\beta_\kappa$ cannot be close rel endpoints to a subsegment of $\alpha_1$: otherwise for sufficiently large $s_i$ we contradict that $h_1 \not\sim_\kappa g_1$.

\textbf{Case 3:} The terminal $g_1^{-s_i/2}$ subsegment of $\beta_1$ cannot be close rel endpoints to a subsegment of $\alpha_\kappa$: otherwise for sufficiently large $s_i$ we contradict that $g_1^\inv \not\sim_1 h_1^\inv$.

\textbf{Case 4:} The initial $g_1^{m_i}$ subsegment of $\beta_1$ cannot be close rel endpoints to a subsegment of $\alpha_\kappa$: otherwise for sufficiently large $m_i$ we contradict that $h_1^\inv \not\sim_1 g_1$.

As a consequence of Cases 1--4 and the fact that $\alpha_1$ is much longer than $\beta_\kappa$ and $\alpha_\kappa$ is much longer than $\beta_1$, it follows that $\alpha_1$ and $\alpha_\kappa$ are somewhat close rel endpoints: the distance between their left endpoints has an upper bound comparable to the lengths of the terminal $g_1^{-m_i/2}$ segment of $\beta_\kappa$ and the terminal $g_1^{-s_i/2}$ segment of $\beta_1$, respectively; and similarly for their right endpoints. But now, using that $s_i > \!\! > m_i,n_i$, it follows that some long subsegment of the $g_1^{-s_i}$ segment of $\beta_1$ is close rel endpoints to the entire $h_1^{-n_i}$ subsegment of $\beta_\kappa$, which for sufficiently large $n_i$ contradicts that $g_1 \not\sim_1 h_1$. 

This completes Step~3.
\end{proof}

\begin{proof}[Step 4: Some quasimorphisms on $N$.] We cite Fujiwara \cite{Fujiwara:H2BHyp} for a method which associates quasimorphisms to loxodromic elements of hyperbolic group actions. Then we cite a result of \BeFuji\ giving some cyclic subgroups on which these quasimorphisms are unbounded and others on which they are bounded. In later steps these results will be applied to the elements $f_1,f_2,\ldots$ constructed in Step~3 and the cyclic subgroups they generate. 

Recall that for each $\kappa \le K_1$ where the action $F \act_\kappa X$ is Schottky, and for each nontrivial $g \in F$, the quasi-axis $A_\kappa(g)$ is a $K,C$-quasigeodesic path in the $\delta$-hyperbolic space $X$. Using the Morse property of quasigeodesics, we fix a constant $B = B(\delta,K,C)$ such that for any $K,C$ quasigeodesic segment $\alpha$ between points $p,q$ in a $\delta$-hyperbolic space, the Hausdorff distance between $\alpha$ and any geodesic $[p,q]$ is at most~$B$. We shall apply this where $\alpha$ is a subsegment of any of the quasi-axes $A_\kappa(g)$.

Henceforth let $W \in [3B,3B+1)$ be the unique integer, as in \BeFuji\ and in \cite{Fujiwara:H2BHyp}.

Given an edge path $\eta = e_1 \cdots e_K$ in $X$ let $\abs{\cdot}=K$ denote its length and $\overline\eta = \bar e_K \cdots \bar e_1$ its reversal. Two subpaths $\eta' = e_{i'} \cdots e_{j'}$ and $\eta'' = e_{i''} \cdots e_{j''}$ are said to \emph{overlap} if $\{i',\ldots,j'\} \intersect \{i'',\ldots,j''\} \ne \emptyset$. 

Consider any edge path $w$ in $X$ such that $\abs{w}>W$. Given another edge path $\alpha$, define a \emph{copy of $w$ in~$\alpha$} to be a subpath of $\alpha$ which is a translate of $w$ by an element of $N$, and define
\begin{align*}
\abs{\alpha}_w &= \text{the maximal number of nonoverlapping copies of $w$ in~$\alpha$} \\
\intertext{Then, given vertices $x,y \in X$, define}
(*) \qquad c_w(x,y) &= d(x,y) - \inf_\alpha \left( \abs{\alpha} - W \abs{\alpha}_w \right) \\
 &= \sup_\alpha \left( W \abs{\alpha}_w - (\abs{\alpha} - d(x,y)) \right) 
\end{align*}
where $\alpha$ varies over all edge paths in $X$ from $x$ to $y$. 

\smallskip

\emph{Remark.} If $\alpha$ is a geodesic then the parenthetical quantity inside the supremum of $(*)$ is equal to $W \abs{\alpha}_w$ which hearkens back to the original method of Brooks in \cite{Brooks:H2bRemarks} computing $H^2_b$ for a free group, where $X$ is the Cayley tree of a free group and $\alpha$ is a geodesic in that tree. When $X$ is a hyperbolic metric space, the method of Fujiwara in \cite{Fujiwara:H2BHyp} allows $\alpha$ to travel through a sequence of nonoverlapping ``wormholes''---copies of $w$---in an attempt to optimize the parenthetical quantities of $(*)$. With this allowance a path $\alpha$ which achieves the optimum need no longer be a geodesic, but at least it is a quasigeodesic (see~\pref{ItemOptimumQG}~below).

\smallskip

Pick once and for all a base point $x_0 \in X$. For each edge path $w$ in $X$ define a function $h_{w} \from N \to \reals$ as follows:
$$(**) \qquad h_{w}(\gamma) = c_w(x_0,\gamma \cdot x_0) - c_{\overline w}(x_0,\gamma \cdot x_0)
$$

\begin{enumeratecontinue}
\item \label{ItemQMData}
Assuming $\abs{w} \ge 2W$:
\begin{enumerate}
\item \label{ItemOptimumQG}
Any edge path $\alpha$ in $X$ which realizes the optimum in $(*)$ is a quasigeodesic with constants $\ell = 2$, $c \ge 0$ depending only on $\delta,K,C$.
\item\label{ItemQMValueBound}
For each $\gamma \in N$ we have 
$$\abs{h_w(\gamma)} \le H(\gamma)
$$
where the constant $H(\gamma)$ is independent of $w$.
\item \label{ItemQMDefectBound}
The function $h_{w}$ is a quasimorphism with defect bounded above by a constant $D$ depending only on $\delta,K,C$.
\end{enumerate}
\end{enumeratecontinue}
Item~\pref{ItemOptimumQG} follows from \cite[Lemma 3.3]{Fujiwara:H2BHyp} which says that $\alpha$ is a quasigeodesic with multiplicative constant $\frac{\abs{w}}{\abs{w} - W} \le 2 = \ell$ and additive constant $\frac{2W\abs{w}}{\abs{w} - W} \le 4W<12B+4=c$. For item~\pref{ItemQMValueBound}, if the edge path $\alpha$ from $x_0$ to $\gamma \cdot x_0$ realizes the optimum in the first term $c_w(x_0,\gamma \cdot c_0)$ of $h_w(\gamma)$ then by \pref{ItemOptimumQG} it follows that 
$$0 \le c_w(x_0,\gamma \cdot x_0) \le \abs{\alpha} \le \ell \, d(x_0,\gamma \cdot x_0) + c
$$
and the same bound holds for the second term $c_{\overline{w}}(x_0,\gamma \cdot x_0)$. Item~\pref{ItemQMDefectBound} follows from \cite[Proposition 3.10]{Fujiwara:H2BHyp} which gives a formula for an upper bound to the defect involving only the quantities $\delta$, $W$, $\ell=2$, and $c=12B+4$ (the latter two numbers $\ell$ and~$c$ being upper bounds for the quantities $\frac{\abs{w}}{\abs{w}-W}$ and $\frac{2 W \abs{w}}{\abs{w}-W}$ as we have just shown).

Next define a quasimorphism $h(f) \from N \to \reals$ for each nontrivial $f \in F$ as follows. First, for each line $L \subset T$ pick a base vertex $x_T(L) \in L$, chosen so that $\xi(x_T(L))$ is a point of $\xi(L)$ closest to $x_0$. Next, for each nontrivial $f \in F$ define $x(f) = \xi(x_T(A_T(f))) \in A(f)$, and notice that $x(f)=x(f^d)$ for all integers $d \ne 0$. Then define $w(f)$ to be an oriented geodesic path in $X$ with initial vertex $x(f)$ and terminal vertex $f \cdot x(f)$. 
Finally define 
$$h(f) = h_{w(f)} \from N \to \reals
$$
Although it need not be true in general that $w(f) \ge 2W$ as needed to apply~\pref{ItemQMData}, as long as $f$ is loxodromic the inequality $\abs{w(f^d)} \ge 2W$ is true for sufficiently large integers $d>0$, which leads to:
\begin{enumeratecontinue}
\item\label{ItemUnbddQM}
There is an integer $d(f)>0$ defined for each loxodromic $f \in N$, and an integer $d(f,f')>0$ defined for each pair $f,f' \in N$ such that $f$ is loxodromic, such that the following hold:
\begin{enumerate}
\item\label{ItemNonNegUnbddPos}
If $f \not\sim f^\inv$ then for all integers $d \ge d(f)$ the values of $h(f^d)$ are non-negative and unbounded on positive elements of the cyclic group $\<f\>$.
\item \label{ItemNotSimPMZero}
If $f \not\sim f'{}^{\pm 1}$ then for all integers $d \ge d(f,f')$ the value of $h(f^d)$ is zero on each element of the cyclic group $\<f'\>$.
\item\label{ItemNonNegOnPos}
If $f \not\sim f'{}^\inv$ then for all integers $d \ge d(f,f')$ the values of $h(f^d)$ are non-negative on positive elements of the cyclic group $\<f'\>$.
\end{enumerate}
\end{enumeratecontinue}
Item~\pref{ItemNonNegUnbddPos}, and item~\pref{ItemNotSimPMZero} when $f'$ is loxodromic, are proved in \BeFuji\ Proposition~5. 

Before continuing, here are a few remarks. First we note that \BeFuji\ Proposition~5 has a hypothesis that $f$ be cyclically reduced, which guarantees that if one chooses $x_0 \in X$ to be the image under $\xi$ of the identity vertex in the Cayley tree of $F$ then $A(f)$ passes through $x_0$; but the proof goes through without change using instead that $A(f)$ passes through $x(f)$. We note also that \BeFuji\ Proposition~5 has an implicit hypothesis $\abs{w(f^d)} > W \ge 3B$, which we arrange by requiring $d$ to be sufficiently large. Finally, we note that item~\pref{ItemNotSimPMZero} is a consequence of item~\pref{ItemNonNegOnPos} (proved in detail below) applied once as stated and once with $f'$ replaced by $f'{}^\inv$.

To prove items~\pref{ItemNotSimPMZero} and~\pref{ItemNonNegOnPos} when $f'$ is not loxodromic, by the WWPD hypothesis the element $f'$ must be elliptic meaning that the diameter of the set $\{f'{}^i(x_0)\}_{i \in \Z}$ is bounded by some constant $\Delta$. If $d$ is sufficiently large then $w(f^d)$ has length greater than $\ell\Delta+c$, which by~\pref{ItemOptimumQG} is an upper bound for the length of any path that realizes the optimum for either term of the expression $h(f^d)(f'{}^i) = c_{w(f^d)}(x_0,f'{}^i(x_0)) - c_{\overline w(f^d)}(x_0,f'{}^i(x_0))$. It follows that $h(f^d)$ is zero on $\<f'\>$. 

\smallskip

\textbf{Remark.} We do not know how to prove \pref{ItemNotSimPMZero} if one were to allow $f'$ to be parabolic (c.f.\ the final paragraph of Section~3 of \BeFuji).

\smallskip

Here is a proof of item~\pref{ItemNonNegOnPos} when $f'$ is loxodromic, which is similar to elements of the proof of \BeFuji\ Proposition~5. If~\pref{ItemNonNegOnPos} fails then there exist integer sequences $d_i \to +\infinity$ and $e_i > 0$ such that the following quantity is negative for all~$i$:
$$h(f^{d_i})(f'{}^{e_i}) = c_{w(f^{d_i})}(f'{}^{e_i}) - c_{\overline w(f^{d_i})}(f'{}^{e_i})
$$
It follows that $c_{\overline w(f^{d_i})}(f'{}^{e_i}) = c_{w(f^{d_i})}(f'{}^{-e_i})$ is positive for all~$i$. The geodesic path $w(f^{d_i})$ has the same endpoints as the $K,C$-quasigeodesic subsegment $\alpha_i \subset A(f)$ with initial endpoint $x(f)$ and terminal endpoint $f^{d_i} \cdot x(f)$, and therefore $w(f^{d_i})$ and $\alpha_i$ are uniformly Hausdorff close rel endpoints. By \pref{ItemOptimumQG} there is an $\ell,c$ quasigeodesic path $\beta_i$ from $x_0$ to $f'{}^{-e_i} \cdot x_0$ realizing the optimum in the definition of the number $c_{w(f^{d_i})}(f'{}^{-e_i})$, and since that number is positive it follows that some subsegment of~$\beta_i$ is a translate of~$w(f^{d_i})$ and therefore is uniformly Hausdorff close rel endpoints to a translate of the $K,C$-quasigeodesic segment~$\alpha_i$. Recalling that $x(f')$ is a point of $A(f')$ closest to $x_0$, letting $[x_0,x(f')]$ denote a geodesic segment in $X$ with the indicated endpoints, and letting $\gamma_i \subset A(f'{}^\inv)$ denote the $K,C$-quasigeodesic subsegment with initial point $x(f')$ and terminal point $f'{}^{-e_i} \cdot x(f')$, the following path is also a uniform quasigeodesic from $x_0$ to $f'{}^{-e_i} \cdot x_0$ and so stays uniformly Hausdorff close rel endpoints to $\beta_i$:
$$[x_0,x(f')] * \gamma_i * \left( f'{}^{-e_i} \cdot [x(f'),x_0] \right)
$$
It follows that some subsegment of the above path is uniformly Hausdorff close rel endpoints to a translate of $\alpha_i$. Noting that $\alpha_i$ is a concatenation of $d_i$ consecutive fundamental domains of the axis $A(f)$, noting also that the prefix $[x_0,x(f')]$ is independent of~$i$, and noting that the suffix $f'{}^{-e_i}[x(f'),x_0]$ has length independent of~$i$, it follows that some subsegment of $\gamma_i$ is uniformly Hausdorff close to a translate of a subpath $\alpha'_i \subset \alpha_i$ consisting of a concatenation of $d_i-d_0$ consecutive fundamental domains of $A(f)$ where $d_0$ is constant. Since $d_i \to +\infinity$ it follows that $\Length(\alpha'_i) \to +\infinity$, from which it follows that $f \sim f'{}^\inv$, a contradiction.

This completes Step 4.
\end{proof}

\begin{proof}[Step 5: Quasimorphisms on $\Gamma$.] \quad We now define a sequence of quasimorphisms on $\Gamma$ with uniformly bounded defects, denoted
$$h_1,h_2,h_3,\ldots \from \Gamma \to \reals
$$
We shall use this sequence in Step 6 to prove Theorem~\ref{ThmWWPDHTwoB}. 

By combining \pref{ItemNotSimThree}, \pref{ItemNotSimOne}, \pref{ItemQMData} and \pref{ItemUnbddQM}, we may choose for each $i \ge 1$ an integer $d_i \ge 1$ such that the following hold:
\begin{enumeratecontinue}
\item\label{ItemQMUnifBddDefect} The function $h(f_i^{d_i}) \from N \to \reals$ is a quasimorphism with defect bounded by a constant $D \ge 0$ independent of $i$ (by \pref{ItemQMDefectBound}).
\item\label{ItemQMUnbddPositive}
The quasimorphism $h(f_i^{d_i})$ is non-negative and unbounded on positive elements of the cyclic group $\<f_i\>$ (by \pref{ItemNotSimThree} with $\kappa=1$, and by~\pref{ItemNonNegUnbddPos}).
\item\label{ItemQMKappaZero}
If $2 \le \kappa \le K$ then the quasimorphism $h(f_i^{d_i})$ is non-negative on positive elements of the cyclic group $\<i_\kappa(f_i)\>$ (by \pref{ItemNotSimThree} with $2 \le \kappa \le K$, and by~\pref{ItemNonNegOnPos}).
\item\label{ItemQMUnbddForIgtJ}
If $i>j$ and $1 \le \kappa \le K$ then the quasimorphism $h(f_i^{d_i})$ is zero on the cyclic group $\<i_\kappa(f_j)\>$ (by \pref{ItemNotSimOne} and~\pref{ItemNotSimPMZero}).
\end{enumeratecontinue}
Note that the constant $d_i$ is made independent of $\kappa$, and in~\pref{ItemQMUnbddForIgtJ} of $j$, by the maximization trick.

%
%
%
%
%
Define functions $h'_i \from N^Q \semidirect \Sym(Q) \to \reals$ and $h_i \from \Gamma \to \reals$ as follows:
\begin{align*}
h'_i(\rho,\phi) &= \sum_{\kappa=1}^K h(f_i^{d_i})(\rho(B_\kappa)), \qquad \text{for} \,\, \rho \in N^Q, \phi \in \Sym(Q) \\
h_i(\mu) &= h'_i(\theta(\mu)) = h'_i(\rho_\mu,\phi_\mu), \qquad \text{for}\,\, \mu \in \Gamma \\
 &= \sum_{\kappa=1}^K h(f_i^{d_i})(\rho_\mu(B_\kappa)) \\
 &= \sum_{\kappa=1}^K h(f_i^{d_i})(i_\kappa(\mu)) \qquad\text{in the special case $\mu \in N$.}
\end{align*}
Note that $h_i(\mu)$ has no dependence on $\phi_\mu$, it depends on $\rho_\mu$ alone. We show:
\begin{enumeratecontinue}
\item\label{ItemQMDefectBoundKD}
Each $h_i \from \Gamma \to \reals$ is a quasimorphism with defect $\le KD$.
\end{enumeratecontinue}
For each $\mu,\nu \in \Gamma$ we have
\begin{align*}
h_i(\mu \cdot \nu) &= h'_i(\theta(\mu \cdot \nu)) = h'_i(\theta(\mu) \cdot \theta(\nu)) = h'_i\left( (\rho_\mu,\phi_\mu) \cdot (\rho_\nu, \phi_\nu) \right) \\
 &= h'_i (\rho_\mu \cdot (\rho_\nu \composed \phi_\mu), \phi_\nu \phi_\mu) \\
 &= \sum_{\kappa=1}^K h(f_i^{d_i})(\rho_\mu(B_\kappa) \cdot \rho_\nu(\phi_\mu(B_\kappa))) \quad\text{(recall that $\rho_\mu(B_\kappa)$, $\rho_\mu(\phi_\mu(B_\kappa)) \in N$)}\\
 &{}^{KD}\!\!= \sum_{\kappa=1}^K h(f_i^{d_i})(\rho_\mu(B_\kappa)) + \sum_{\kappa=1}^K h(f_i^{a_i})(\rho_\nu(\phi_\mu(B_\kappa))) \\
\intertext{where the symbol ${}^{KD}\!\!=$ means that the difference of the two sides has absolute value bounded by~$KD$; this holds because of~\pref{ItemQMUnifBddDefect}. Since $\phi_\mu \in \Sym(Q)$ permutes the set of cosets $Q = \{B_1,\ldots,B_K\}$ of $N$ in $\Gamma$, by rewriting the second term on the right hand side of the last equation we obtain:}
 &= \sum_{\kappa=1}^K h(f_i^{a_i})(\rho_\mu(B_\kappa)) + \sum_{\kappa=1}^K h(f_i^{a_i})(\rho_\nu(B_\kappa)) \\ &= h_i(\mu) + h_i(\nu)
\end{align*}
This completes Step~5.
\end{proof}

\begin{proof}[Step 6: The proof of the Global WWPD Theorem
] 
To 
embed $\ell^1 \inject H^2_b(\Gamma;\reals)$, consider 
a sequence of real numbers $(t)=(t_1,t_2,t_3,\ldots)$ with $\sum \abs{t_i} < \infinity$. Define the function $h_{(t)} \from \Gamma \to \reals$ by 
$$h_{(t)} = t_1 h_1 + t_2 h_2 + t_3 h_3 + \cdots
$$ 
By \pref{ItemQMDefectBoundKD}, each $h_i$ is a quasimorphism with defect $\le KD$, and it follows that $h_{(t)}$ is a convergent series defining a quasimorphism with defect $\le KD \sum \abs{t_i}$. 

The map $(t) \mapsto h_{(t)}$ evidently defines a linear map from the vector space $\ell^1$ to the vector space of quasimorphisms $\Gamma \mapsto \reals$, and so it remains to show that if $(t)$ is nonzero in $\ell^1$ then $h_{(t)}$ is unbounded and has unbounded difference with every homomorphism $\Gamma \mapsto \reals$. Letting $j \ge 1$ be the least integer such that $t_j \ne 0$ we have:
\begin{align*}
h_{(t)}(f_j^d) &= \underbrace{\sum_{i < j} t_i h_i(f_j^d)}_{=0} + t_j h_j(f_j^d) + \sum_{i > j} t_i h_i(f_j^d)\\
 &= t_j \sum_{\kappa=1}^K h(f_j^{d_j})(i_\kappa(f_j^d)) + \sum_{i > j} t_i \left( \sum_{\kappa=1}^K \underbrace{h(f_i^{d_i})(i_\kappa(f_j^d))}_{=0 \, \text{by \pref{ItemQMUnbddForIgtJ}}} \right) \\
 &= t_j \left(h(f_j^{d_j})(f_j^d) + \sum_{\kappa=2}^K h(f_j^{d_j}) (i_\kappa(f_j^d)) \right)
\end{align*}
Regarding the quantity in the big parentheses on the last line, as $d>0$ varies the summand $h(f_j^{d_j})(f_j^d)$ is non-negative and unbounded by~\pref{ItemQMUnbddPositive}, and the other summands with $2 \le \kappa \le K$ are non-negative by \pref{ItemQMKappaZero}, so the whole quantity is unbounded. This shows that $h_{(t)}(f^d_j)$ is unbounded. Also, since $f_j^d \in F \subgroup [N,N] \subgroup [\Gamma,\Gamma]$, the value of every homomorphism $\Gamma\mapsto\reals$ on $f_j^d$ is zero, and so $h_{(t)}$ has unbounded difference with every such homomorphism. 
\end{proof}

\bibliographystyle{amsalpha} 
\bibliography{mosher} 

\def\cprime{$'$} \def\cprime{$'$}
\providecommand{\bysame}{\leavevmode\hbox to3em{\hrulefill}\thinspace}
\providecommand{\MR}{\relax\ifhmode\unskip\space\fi MR }
\providecommand{\MRhref}[2]{%
  \href{http://www.ams.org/mathscinet-getitem?mr=#1}{#2}
}
\providecommand{\href}[2]{#2}
\begin{thebibliography}{{Alo}91}

\bibitem[{Alo}91]{ABCFLMSS}
{Alonso, J. M. and Brady, T. and Cooper, D. and Ferlini, V. and Lustig, M. and
  Mihalik, M. and Shapiro, M. and Short, H.}, \emph{Notes on word hyperbolic
  groups}, Group theory from a geometrical viewpoint ({T}rieste, 1990)
  (H.~Short, ed.), World Sci. Publ., River Edge, NJ, 1991, pp.~3--63.

\bibitem[BBF15]{BBF:MCGquasitrees}
M.~Bestvina, K.~Bromberg, and K.~Fujiwara, \emph{Constructing group actions on
  quasi-trees and applications to mapping class groups}, Publ. Math. Inst.
  Hautes \'Etudes Sci. \textbf{122} (2015), 1--64.

\bibitem[BBF16]{BBF:SCLonMCG}
\bysame, \emph{Stable commutator length on mapping class groups}, Ann. Inst.
  Fourier \textbf{66} (2016), no.~3, 871--898, arXiv:1306:2394.

\bibitem[BF02]{BestvinaFujiwara:bounded}
M.~Bestvina and K.~Fujiwara, \emph{Bounded cohomology of subgroups of mapping
  class groups}, Geom. Topol. \textbf{6} (2002), 69--89 (electronic).

\bibitem[BHS14]{BHS:Hierarchy}
J.~Behrstock, M.~Hagen, and A.~Sisto, \emph{{Hierarchically hyperbolic spaces
  I: curve complexes for cubical groups}}, arXiv:1412.2171, 2014.

\bibitem[Bow08]{Bowditch:tight}
B.~Bowditch, \emph{Tight geodesics in the curve complex}, Invent. Math.
  \textbf{171} (2008), no.~2, 281--300.

\bibitem[Bro81]{Brooks:H2bRemarks}
R.~Brooks, \emph{Some remarks on bounded cohomoogy}, Rieman surfaces and
  related topics: Proceedings of the 1978 Stony Brook Conference, Ann. of Math.
  Stud., vol.~97, Princeton Univ. Press, 1981, pp.~53--63.

\bibitem[Bro82]{Brown:cohomology}
K.~Brown, \emph{Cohomology of groups}, Graduate Texts in Math., vol.~87,
  Springer, 1982.

\bibitem[BS84]{BrooksSeries:H2bSurface}
R.~Brooks and C.~Series, \emph{Bounded cohomology for surface roups}, Topology
  \textbf{23} (1984), no.~1, 29--36.

\bibitem[CDP90]{CDP}
M.~Coornaert, T.~Delzant, and A.~Papadopoulos, \emph{G\'eom\'etrie et th\'eorie
  des groupes}, Lec. Notes Math., vol. 1441, Springer-Verlag, 1990.

\bibitem[EF97]{EpsteinFujiwara}
D.~B.~A. Epstein and K.~Fujiwara, \emph{The second bounded cohomology of
  word-hyperbolic groups}, Topology \textbf{36} (1997), no.~6, 1275--1289.

\bibitem[Fuj98]{Fujiwara:H2BHyp}
K.~Fujiwara, \emph{{The second bounded cohomology of a group acting on a
  Gromov-hyperbolic space}}, Proc. London Math. Soc. (3) \textbf{76} (1998),
  no.~1, 70--94.

\bibitem[Fuj00]{Fujiwara:H2bFreeProduct}
\bysame, \emph{The second bounded cohomology of an amalgamated free product of
  groups}, Trans. AMS \textbf{352} (2000), no.~3, 1113--1129.

\bibitem[GdlH91]{GhysHarpe:afterGromov}
E.~Ghys and P.~de~la Harpe, \emph{Infinite groups as geometric objects (after
  {G}romov)}, Ergodic theory, symbolic dynamics and hyperbolic spaces (Trieste,
  1989) (T.~Bedford, M.~Keane, and C.~Series, eds.), Oxford Univ. Press, 1991,
  pp.~299--314.

\bibitem[Gro87]{Gromov:hyperbolic}
M.~Gromov, \emph{Hyperbolic groups}, Essays in group theory (S.~Gersten, ed.),
  MSRI Publications, vol.~8, Springer, 1987.

\bibitem[Ham17]{Hamann:GroupActions}
M.~Hamann, \emph{Group actions on metric spaces: fixed points and free
  subgroups}, Abh. Math. Semin. Univ. Hambg. \textbf{87} (2017), no.~2,
  245--263.

\bibitem[HM15]{HandelMosher:BddCohomologyI}
M.~Handel and L.~Mosher, \emph{{Hyperbolic actions and 2nd bounded cohomology
  of subgroups of $\mathsf{Out}(F_n)$ Part I: Infinite lamination subgroups}},
  arXiv:1511.06913, 2015.

\bibitem[HM17]{HandelMosher:BddCohomologyII}
\bysame, \emph{{Hyperbolic actions and 2nd bounded cohomology of subgroups of
  $\mathsf{Out}(F_n)$ Part II: Finite lamination subgroups}}, arXiv:1702.08050,
  2017.

\bibitem[Osi16]{Osin:AcylHyp}
D.~Osin, \emph{Acylindrically hyperbolic groups}, Trans. AMS \textbf{368}
  (2016), no.~2, 851--888.

\bibitem[Wel76]{Wells:WreathProduct}
C.~Wells, \emph{Some applications of the wreath product construction}, Amer.
  Math. Monthly \textbf{83} (1976), no.~5, 317--338.

\end{thebibliography}

\end{document}